\documentclass[11pt]{amsart}
\usepackage[babel]{csquotes}
\usepackage{enumitem}
\usepackage{amsmath,amsthm,amssymb,mathrsfs,amsfonts,verbatim,enumitem,color,leftidx,mathtools}
\usepackage{mathabx}
\usepackage{etoolbox} 
\usepackage{bbm}
\usepackage[all,tips]{xy}
\usepackage{graphicx,ifpdf}
\usepackage{stmaryrd}
\usepackage{amssymb}
\usepackage{dsfont}

\ifpdf
\fi
\usepackage[colorlinks]{hyperref}
\hypersetup{
linkcolor=blue,          % color of internal links (change box color with linkbordercolor)
citecolor=green,        % color of links to bibliography
}

\usepackage{tikz}%pictures drawn in latex

\newtheorem{thm}{Theorem}[section]
\newtheorem{lem}[thm]{Lemma}
\newtheorem{cor}[thm]{Corollary}
\newtheorem{prop}[thm]{Proposition}

\theoremstyle{definition}

\theoremstyle{remark}

\numberwithin{equation}{section}
\definecolor{esperance}{rgb}{0.0,0.5,0.0}

\newcommand{\be}{\mathbf{e}}

\newcommand{\bbr}{\mathbf{r}}

\newcommand{\bu}{\mathbf{u}}
\newcommand{\bv}{\mathbf{v}}
\newcommand{\bw}{\mathbf{w}}

\DeclareMathOperator{\diam}{diam}

\DeclareMathOperator{\Span}{Span}

\DeclareMathOperator{\vol}{vol}
\DeclareMathOperator{\cov}{cov}

\newcommand\ov[1]{\overline{#1}}

\newcommand{\lam}{\lambda}
\newcommand{\Lam}{\Lambda}
\newcommand{\eps}{\epsilon}

\newcommand{\vphi}{\varphi}
\newcommand{\cE}{\mathcal{E}}
\newcommand{\cF}{\mathcal{F}}

\newcommand{\cK}{\mathcal{K}}
\newcommand{\cL}{\mathcal{L}}

\newcommand{\cQ}{\mathcal{Q}}

\newcommand{\cS}{\mathcal{S}}
\newcommand{\cT}{\mathcal{T}}
\newcommand{\cU}{\mathcal{U}}

\newcommand{\bR}{\mathbb{R}}
\newcommand{\bZ}{\mathbb{Z}}
\newcommand{\bQ}{\mathbb{Q}}

\newcommand{\bN}{\mathbb{N}}

\newcommand{\SL}{\operatorname{SL}}

\newcommand{\Sing}{\operatorname{Sing}}
\newcommand\wt[1]{\widetilde{#1}}
\newcommand\wh[1]{\widehat{#1}}

\newcommand\pa[1]{\left(#1\right)}

\newcommand\diag[1]{\operatorname{diag}\left(#1\right)}
\newcommand\mb[1]{\mathbf{#1}}

\newcommand\tb[1]{\textbf{#1}}
\newcommand\mat[1]{\pa{\begin{matrix}#1\end{matrix}}}

\newcommand{\pr}{\operatorname{\mb{pr}}}

\newcommand{\onto}{\xymatrix{\ar@{>>}[r]&}}
\newcommand{\eq}[1]
{
\begin{equation*}
{#1}
\end{equation*}
}
\newcommand{\eqlabel}[2]
{
\begin{equation}
{#2}\label{#1}
\end{equation}
}
\makeatletter
\newcommand*{\rom}[1]{\expandafter\@slowromancap\romannumeral #1@}
\makeatother

\begin{document}

\title[Hausdorff dimension of weighted singular vectors]{On a lower bound of Hausdorff dimension of weighted singular vectors}

\date{}

\author{Taehyeong Kim}
\address{Taehyeong ~Kim. Einstein Institute of Mathematics, Hebrew University of Jerusalem, {\it taehyeong.kim@mail.huji.ac.il}}
%\address{Taehyeong ~Kim. Department of Mathematical Sciences, Seoul National University, {\it kimth@snu.ac.kr}}

\author{Jaemin Park}
\address{Jaemin ~Park. Department of Mathematical Sciences, Seoul National University, {\it woalsee@snu.ac.kr}}

% \thanks will become a 1st page footnote.
\thanks{}

%\author{}


\def\thefootnote{}
   %%d 
\def\thefootnote{\arabic{footnote}}
\setcounter{footnote}{0}

\begin{abstract}
Let $w=(w_1,\dots,w_d)$ be a $d$-tuple of positive real numbers such that $\sum_{i}w_i =1$ and $w_1\geq \cdots \geq w_d$.
A $d$-dimensional vector $x=(x_1,\dots,x_d)\in\bR^d$ is said to be $w$-singular if for every $\eps>0$ there exists $T_0>1$ such that for all $T>T_0$
the system of inequalities
\eq{
\max_{1\leq i\leq d}|qx_i - p_i|^{\frac{1}{w_i}} < \frac{\eps}{T} \quad\text{and}\quad 0<q<T
}
has an integer solution $(\mb{p},q)=(p_1,\dots,p_d,q)\in \bZ^d \times \bZ$.
We prove that the Hausdorff dimension of the set of $w$-singular vectors in $\bR^d$ is bounded below by $d-\frac{1}{1+w_1}$. Our result partially extends the previous result of Liao et al. [\emph{Hausdorff dimension of weighted singular vectors in $\bR^2$}, J. Eur. Math. Soc. \textbf{22} (2020), 833-875].  
\end{abstract}

\maketitle

% Section : Introduction %%%%%%%%%%%%%%%%%%%%%%%%%%%%%%%%%%%%%%%%%%%%%%%%%%%%%%%%%%%%

\section{Introduction}
In 1926, Khintchine \cite{Khi26} introduced the notion of singularity in the sense of Diophantine approximation.
Recall that a vector $x=(x_1,\dots,x_d)\in\bR^d$ is said to be \textit{singular} if, for every $\eps>0$, there exists $T_0>1$ such that for all $T>T_0$
the system of inequalities 
\eqlabel{unweightedsing}{
\max_{1\leq i\leq d}|qx_i - p_i|^d < \frac{\eps}{T} \quad\text{and}\quad 0<q<T
}
admits an integer solution $(\mb{p},q)=(p_1,\dots,p_d,q)\in \bZ^d \times \bZ$.

The name \textit{singular} is derived from the fact that the set of singular vectors is a Lebesgue null set. 
On the other hand, the computation of the Hausdorff dimension of the set of singular vectors, or more generally singular matrices, has been a challenge 
until the breakthrough \cite{DFSU} using a variational principle in the parametric geometry of numbers. 
Historically, the first breakthrough was made in \cite{C11}, where it is proved that the Hausdorff dimension of the set of $2$-dimensional singular vectors is $4/3$.
This result was extended to the $d$-dimensional singular vectors in \cite{CC16}, where the authors proved that the set of $d$-dimensional singular vectors has Hausdorff dimension $d^2/(d+1)$.
For general $m\times n$ singular matrices, it was proved in \cite{KKLM} that the Hausdorff dimension of $m\times n$ singular matrices is at most 
$mn(1-\frac{1}{m+n})$ , and finally, it was shown in \cite{DFSU} that the upper bound is sharp.

In this article, we consider the weighted version of the singularity as follows:
Let $w=(w_1,\dots,w_d)\in \bR_{>0}^d$ be a $d$-tuple of positive real numbers such that $\sum_{i}w_i =1$ and $w_1\geq \cdots \geq w_d$.
We say that a vector $x=(x_1,\dots,x_d)\in\bR^d$ is \textit{$w$-singular} if for every $\eps>0$ there exists $T_0>1$ such that for all $T>T_0$
the system of inequalities
\eqlabel{weightedsing}{
\max_{1\leq i\leq d}|qx_i - p_i|^{\frac{1}{w_i}} < \frac{\eps}{T} \quad\text{and}\quad 0<q<T
}
admits an integer solution $(\mb{p},q)=(p_1,\dots,p_d,q)\in \bZ^d \times \bZ$. Denote by $\Sing(w)$ the set of $w$-singular vectors in $\bR^d$.
Here and hereafter, we always assume that the weight vector $w$ satisfies the above assumption.

Note that if all weights have the same value, that is, $w_i=1/d$ for all $i=1,\dots,d$, then the weighted system \eqref{weightedsing} is nothing but the unweighted system
\eqref{unweightedsing}. 

In the weighted setting, it was shown in \cite{LSST} that the set of $2$-dimensional $w$-singular vectors has Hausdorff dimension $2-\frac{1}{1+w_1}$.
The aim of the present article is to extend this $2$-dimensional result to higher dimensional cases regarding the lower bound of the Hausdorff dimension. Our main theorem is as follows.

\begin{thm}\label{mainthm}
For $d\geq 2$, 
\[
\dim \Sing(w) \geq d -\frac{1}{1+w_1},
\] where $\dim$ refers to the Hausdorff dimension.
%the Hausdorff dimension of $\Sing(w)$ is at least $d-\frac{1}{1+w_1}$.
\end{thm}

One of the main ingredients of the proof of Theorem \ref{mainthm} is Dani's correspondence, which means that $w$-singular vectors correspond 
to certain divergent trajectories in the space $\cL_{d+1}$ of unimodular lattices in $\bR^{d+1}$. More precisely, 
let $a_t:=\diag{e^{w_1 t},\dots, e^{w_d t},e^{-1}}\in \SL_{d+1}(\bR)$ and let
\[h(x):=\mat{I_d & x \\ 0 & 1}\in \SL_{d+1}(\bR) \quad \text{ for }  x\in\bR^d,\]
where $I_d$ is the $d\times d$ identity matrix.
Then $x$ is $w$-singular if and only if the diagonal orbit $\left(a_t h(x) \bZ^{d+1}\right)_{t\geq 0}$ is divergent.

Our method for the proof of Theorem \ref{mainthm} is basically an extension of the method in \cite{LSST}, hence we also have the following result
as in \cite[Theorem 1.5]{LSST}.
\begin{thm}\label{mainthm_2}
For any $\Lambda\in\cL_{d+1}$ and any nonempty open subset $U$ in $\bR^{d}$, the Hausdorff dimension of the set
\[
\{x\in U : \left(a_t h(x) \Lambda \right)_{t\geq 0} \text{ is divergent}\}
\] is at least $d-\frac{1}{1+w_1}$.
\end{thm}
Theorem \ref{mainthm_2} implies the following corollary as in \cite[Corollary 1.6]{LSST} (see also \cite{GS,AGMS}).
\begin{cor}\label{cor_1}
The Hausdorff dimension of the set
\[
\{\Lambda\in\cL_{d+1}: \left(a_t \Lambda \right)_{t\geq 0} \text{ is divergent}\}
\] is at least $\dim\SL_{d+1}(\bR)-\frac{1}{1+w_1}=(d+1)^2 -1 -\frac{1}{1+w_1}$.
\end{cor}

Recently, Solan \cite{Sol} established a variational principle in the parametric geometry of numbers for general flows. Following his notations,
we consider the following two subgroups:
\[
\begin{split}
H&=\{g\in\SL_{d+1}(\bR): a_{-t}ga_{t}\to I_{d+1}\text{ as } t\to \infty\},\\
H'&=\{h(x)\in\SL_{d+1}(\bR) : x\in \bR^d \}.
\end{split}
\] 
Note that $H$ is the unstable horospherical subgroup of $a_1$. In the unweighted setting ($w_1=\dots=w_d$), the two subgroups $H$ and $H'$ are the same, but in general, $H$ is bigger than $H'$.  
One of the applications of the variational principle for general flows in \cite{Sol} is to give an upper bound of the Hausdorff dimension of the set
\[
\Sing(H,\Lambda ; a_t)=\{h\in H : \left(a_t h\Lambda \right)_{t\geq 0}\text{ is divergent}\}.
\]
More precisely, \cite[Corollary 2.34]{Sol} implies that the Hausdorff dimension of $\Sing(H,\Lambda ; a_t)$ 
is at most $\mathrm{dim}~H - \frac{1}{1+w_1}$.
On the other hand, Theorem \ref{mainthm_2} implies that  the Hausdorff dimension of $\Sing(H,\Lambda ; a_t)$ is at least 
$\dim H - \frac{1}{1+w_1}$, hence we have the following corollary.
\begin{cor}
The Hausdorff dimension of $\Sing(H,\Lambda ; a_t)$
is $\dim H - \frac{1}{1+w_1}$.
\end{cor}

\textit{Structure of the paper.} In Section \ref{sec2}, we first recall fractal structures and self-affine structures. Then we estimate the lower bound of the Hausdorff dimension of the associated fractal set. In Section \ref{sec3}, we generalize the lattice point counting results in \cite[Section 3]{LSST} to higher dimensional cases. In Section \ref{sec4}, we construct a fractal set contained in the set of weighted singular vectors and prove Theorem \ref{mainthm} by estimating the Hausdorff dimension of the fractal set. \\

\textit{Convention.} In what follows, the notation $A \ll B$ means that there exists a constant $C$ (the \textit{implied constant}) such that $A \leq CB$. The notation $A \asymp B$ means $A \ll B\ll A$.
When we write $A \ll_{D} B$ or $A \asymp_D B$, we mean that the implied constant depends on $D$. 

\vspace{5mm}
\tb{Acknowledgments}. 
We express our gratitude to Seonhee Lim for her constant help and encouragement.
We also thank Fr\'{e}d\'{e}ric Paulin for providing helpful comments, the anonymous referee for carefully reading this article and the suggestions that improved the exposition.

TK and JP were supported by the National Research Foundation of Korea under Project Number NRF-2020R1A2C1A01011543. TK was supported by the National Research Foundation of Korea under Project Number NRF-2021R1A6A3A13039948.
JP was supported by Industrial and Mathematical Data Analytics Research Center under Project Number 2022R1A5A6000840.

% Section : Fractal structure and Hausdorff dimension %%%%%%%%%%%%%%%%%%%%%%%%%%%%%%%%%%%%%%%%%%%%%%%%%%%%%%%%%%%%

\section{Fractal sutructure and Hausdorff dimension}\label{sec2}

% subsection : Fractal structure %%%%%%%%%%

\subsection{Fractal structure}

A \textit{tree} $\cT$ is a connected graph without cycles.
If we take a vertex $\tau_0$ and fix it (we call it a \textit{root}), then $\cT$ is a \textit{rooted tree}.
In this paper, we identify $\cT$ with the set of vertices of $\cT$.
It can be checked directly from the definition of $\cT$ that any $\tau \in \cT$ can be joined to $\tau_0$ by a unique geodesic edge path.
We define the \textit{height} of $\tau$ as the length of the geodesic edge path joining $\tau$ and $\tau_0$. Denote the set of vertices of height $n$ by $\cT_n$.
For any $\tau \in \cT_n$, there exists a unique $\tau_{n-1} \in \cT_{n-1}$ such that $\tau$ and $\tau_{n-1}$ are adjacent.
In this case, we say that $\tau$ is a \textit{son} of $\tau_{n-1}$. Denote the set of all sons of $\tau_{n-1}$ by $\cT(\tau_{n-1})$.
The \textit{boundary} of $\cT$, denoted by $\partial \cT$, is the set of all sequences $\{\tau_n\} = \{\tau_n\}_{n\in\bN\cup\{0\}}$ where $\tau_n$ is a son of $\tau_{n-1}$ for all $n\in\bN$.

A \textit{fractal structure} on $\bR^d$ is a pair $(\cT,\beta)$ where $\cT$ is a rooted tree and $\beta$ is a map from $\cT$ to the set of nonempty compact subsets of $\bR^d$.
The \textit{fractal} associated to $(\cT,\beta)$ is the set 
\eq{ \cF(\cT,\beta) = \bigcup_{ \{\tau_n\} \in \partial \cT } \bigcap_{n=0}^\infty \beta(\tau_n). }

A fractal structure $(\cT,\beta)$ is said to be \textit{regular} if it satisfies the followings:
\begin{itemize}
\item each vertex of $\cT$ has at least one son;
\item if $\tau$ is a son of $\tau'$,  then $\beta(\tau) \subset \beta(\tau')$;
\item for any $\{\tau_n\} \in \partial \cT$, $\diam{\beta(\tau_n)} \to 0$ as $n \to \infty$.
\end{itemize}

% subsection : Self-affine structure and lower bound %%%%%%%%%%

\subsection{Self-affine structure and lower bound}

A \textit{self-affine structure} on $\bR^d$ is a fractal structure $(\cT,\beta)$ on $\bR^d$ such that for $\tau \in \cT$ the compact subset $\beta(\tau)$ of $\bR^d$ is given by a $d$-dimensional rectangle with size $L^{(1)}(\tau) \times \cdots \times L^{(d)}(\tau)$.
A self-affine structure is \textit{regular} if it is a regular fractal structure.

%We deal with a regular self-affine structure on $\bR^d$ that associates to sequences $\{\rho_n\}, \{C_n\}, \{L_n^{(j)}\}$ for $j=1,\dots,d$ of positive real numbers indexed by $\bN \cup \{0\}$ following \cite[Theorem 2.1]{LSST}.
The following theorem is a generalization of \cite[Theorem 2.1]{LSST} for $d$-dimensional self-affine structures.
\begin{thm}\label{FHThm1}
Let $(\cT,\beta)$ be a regular self-affine structure on $\bR^d$ associated to the sequences $\{\rho_n\}, \{C_n\}, \{L_n^{(j)}\}$ $(j=1,\dots,d)$ of positive real numbers indexed by $\bN \cup \{0\}$ with the following properties.
\begin{enumerate}
\item\label{thm2pro1} The sequence $\{L_n^{(j)}\}$ is decreasing in $n\in \bN \cup \{0\}$ for each $j=1,\dots,d$. 
\item\label{thm2pro2} There exists $1\leq \ell < d$ such that
  	\eq{L_n^{(1)}=\cdots=L_n^{(\ell)} < L_n^{(\ell+1)}\leq \cdots \leq L_{n}^{(d)} \quad\text{and}\quad L^{(j)}(\tau) = L_n^{(j)}}
for all $n\in\bN \cup \{0\}$, $j=1,\dots,d$, and $\tau\in\cT_n$;
\item\label{thm2pro3} $C_0=1$ and $\#\cT(\tau) \geq C_n$ for all $n\in\bN$ and $\tau\in\cT_{n-1}$;
\item\label{thm2pro4} $\rho_n \leq 1$ for all $n\in\bN$ and 
\eq{ \textrm{dist}(\beta(\tau),\beta(\kappa)) \geq \rho_{n+1}L_n^{(1)} }
for all $\tau_n \in \cT_n$ and distinct $\tau,\kappa \in \cT(\tau_n)$.
\end{enumerate}
We denote
\[
\begin{split}
P_n &= \prod_{i=0}^n C_i, \\
D_n &= \max\{i \geq n : L_i^{(d)} \geq L_n^{(1)} \},\\
s &= \sup \left\{ t>0 : \lim_{n \to \infty} \frac{\log (P_n (L_n^{(1)})^t \rho_{n+1}^{t} \cdot \prod_{i=n+1}^{D_n} \rho_i^\ell C_i) }{ \max \{D_n-n,1 \} } =\infty \right\}.
\end{split}
\]
%Let
%\[
%
%s = \sup \left\{ t>0 : \lim_{n \to \infty} \frac{\log (P_n (L_n^{(1)})^t \rho_{n+1}^{t} \cdot \prod_{i=n+1}^{D_n} \rho_i^l C_i) }{ \max \{D_n-n,1 \} } =\infty \right\}.
%\]
If $s>d-\ell$, then $\dim \cF(\cT,\beta) \geq s$.
\end{thm}

Using Theorem \ref{FHThm1}, we obtain the following corollary which is a generalization of \cite[Corollaries 2.3 and 2.4]{LSST} for $d$-dimensional self-affine structures.
\begin{cor}\label{FHCor24}
With the notations in Theorem \ref{FHThm1},
suppose that there exists $k,n_0 \in \bN$ such that for all $n \geq n_0$ the followings hold:
\begin{enumerate}[label=(\roman*)]
\item\label{FHCorA1} $\frac{ L_{kn}^{(d)} }{ L_{kn-1}^{(d)} } \leq \frac{ L_n^{(1)} }{L_{n-1}^{(1)}} \text{ and } L_{kn_0-1}^{(d)} < L_{n_0-1}^{(1)},$
\item\label{FHCorA2} $e^{n/k} \leq C_n \leq e^{kn},$
\item\label{FHCorA3} $e^{-kn} \leq \rho_n \leq e^{-n/k},$
\item\label{FHCorA4} $\rho_n^\ell C_n \prod_{j=\ell+1}^d L_n^{(j)}/L_{n-1}^{(j)} \geq n^{-k}.$
\end{enumerate}
If the limit
\eqlabel{Eq_limit}{ 
\lim_{n \to \infty} \frac{ \log \left(  C_n\prod_{j=\ell+1}^d L_n^{(j)}/L_{n-1}^{(j)} \right)  }{ -\log \left( L_n^{(1)} / L_{n-1}^{(1)} \right)} 
}
exists and is equal to $r>0$, then $\dim \cF(\cT,\beta) \geq d-\ell+r$.
\end{cor}

\begin{proof}[Proof of Corollary \ref{FHCor24}]
By the assumptions \ref{FHCorA3} and \ref{FHCorA4}, since the sequence $\{L_n^{(j)}\}$ is decreasing in $n\in\bN \cup \{0\}$ for each $j=1,\dots,d$, 
we have
\[
\log \left( C_n \prod_{j=\ell+1}^d L_n^{(j)}/L_{n-1}^{(j)} \right) = O(n), 
\]
which by the existence of \eqref{Eq_limit} implies that $ -\log \left( L_n^{(1)} / L_{n-1}^{(1)} \right) \to \infty$ as $n\to \infty$.
Hence, by the identity
\[
\frac{ \log \left(  P_n \prod_{j=\ell+1}^d L_n^{(j)} \right) }{ -\log L_n^{(1)} }
= \frac{\log \left(C_0 \prod_{j=\ell+1}^d L_0^{(j)}\right) + \sum_{i=1}^n \log \left(C_i \prod_{j=\ell+1}^{d}L_i^{(j)}/L_{i-1}^{(j)}\right)}{-\log L_0^{(1)}-\sum_{i=1}^{n}\log\left(L_{i}^{(1)}/L_{i-1}^{(1)}\right)},
\] and the Stolz--Ces\`{a}ro theorem,
we have
\[
\lim_{n \to \infty} \frac{ \log \left(C_n \prod_{j=\ell+1}^d L_n^{(j)}/L_{n-1}^{(j)}  \right)  }{ -\log \left( L_n^{(1)} / L_{n-1}^{(1)} \right)}
= \lim_{n\to\infty} \frac{ \log \left(  P_n \prod_{j=\ell+1}^d L_n^{(j)} \right) }{ -\log L_n^{(1)} }.
\]
Let us denote 
\[s = d-\ell +r = d-\ell + \lim_{n\to\infty} \frac{ \log \left(  P_n \prod_{j=\ell+1}^d L_n^{(j)} \right) }{ -\log L_n^{(1)} } . \]
By the regularity of the given self-affine structure $(\cT,\beta)$, we have that $L_n^{(1)} \to 0$ as $n\to \infty$, which implies 
\eqlabel{newdef_s}{
s = \sup \left\{ t>0 : \lim_{n\to\infty} P_n (L_n^{(1)})^t  \prod_{j=\ell+1}^d  \frac{ L_n^{(j)} }{ L_n^{(1)} } = \infty \right\}.
}

We will show that $\dim \cF(\cT,\beta) \geq s$ using the equality \eqref{newdef_s}.
Recall that
$D_n = \max \left\{ i \geq n : L_i^{(d)} \geq L_n^{(1)} \right\}$.
Since the sequence $\{L_n^{(j)}\}$ is decreasing in $n\in\bN \cup \{0\}$ for each $j=1,\dots,d$, we have
\begin{alignat*}{3}
L_{kn}^{(d)}
& = L_{kn_0-1}^{(d)} \prod_{i=kn_0}^{kn} \frac{ L_i^{(d)} }{ L_{i-1}^{(d)} } 
\leq L_{kn_0-1}^{(d)} \prod_{i=n_0}^{n} \frac{ L_{ki}^{(d)} }{ L_{ki-1}^{(d)} }  \\
& \leq L_{kn_0-1}^{(d)}  \prod_{i=n_0}^{n} \frac{ L_{i}^{(1)} }{ L_{i-1}^{(1)} } \quad &&\text{by assumption } \ref{FHCorA1}\\
& = L_{kn_0-1}^{(d)}  \frac{ L_n^{(1)} }{ L_{n_0-1}^{(1)} } \\
& \leq L_n^{(1)}  \quad &&\text{by assumption } \ref{FHCorA1}.
\end{alignat*}
Hence we have $D_n \leq kn$.

Given $t>0$, $\eps>0$, it follows from the assumptions \ref{FHCorA2}, \ref{FHCorA3}, and $D_n \leq kn$ that 
$\rho_{D_n+1}^{\ell} C_{D_n+1} \leq e^{k(D_n+1)}\leq e^{k(kn+1)}$.
Since $P_n^\eps = (\prod_{i=0}^n C_i )^\eps \gg e^{\frac{n(n+1)\eps}{2k}}$ and the implied constant is independent of $n$, we have
\eqlabel{eq211}{
\rho_{D_n+1}^{\ell} C_{D_n+1} \leq P_n^\eps}
for all large enough $n\geq 1$.
Similarly, it follows from the assumptions \ref{FHCorA3}, \ref{FHCorA4}, and $D_n \leq kn$ that 
\[
\begin{split}
\rho_{n+1}^t \prod_{i=n+1}^{D_n+1} \left( \rho_i^\ell C_i \prod_{j=\ell+1}^d \frac{L_i^{(j)}}{L_{i-1}^{(j)}} \right) &\geq
e^{-tk(n+1)}\prod_{i=n+1}^{D_n+1}i^{-k} \geq e^{-tk(n+1)}(kn+1)^{-k(kn-n)}\\
&= e^{-tk(n+1)-k(kn-n)\log (kn+1)}. 
\end{split}
\]
The inequality $P_n^{-\eps} \ll e^{-\frac{n(n+1)\eps}{2k}}$ implies that
\eqlabel{eq212}{
\rho_{n+1}^t \prod_{i=n+1}^{D_n+1} \left( \rho_i^\ell C_i \prod_{j=\ell+1}^d \frac{L_i^{(j)}}{L_{i-1}^{(j)}} \right) \geq P_n^{-\eps}
}
for all large enough $n\geq 1$.

%\begin{equation}\label{eq21}
%\min \left\{ (\rho_{D_n+1}^\ell C_{D_n+1})^{-1}, \rho_{n+1}^t \prod_{i=n+1}^{D_n+1} \middle( \rho_i^\ell C_i \prod_{j=l+1}^d \frac{ L_i^{(j)} }{ L_{i-1}^{(j)} } \middle) \right\} \geq P_n^{-\eps}
%\end{equation}
%for all sufficiently large $n$.

Fix a real number $t$ with $d-\ell < t < s$ and take sufficiently small $\eps$ such that $d-\ell < t/(1-3\eps) < s.$
By the equality \eqref{newdef_s}, we have
\begin{equation}\label{eq22}
\lim_{n\to\infty} P_n (L_n^{(1)})^{t/(1-3\eps)} \prod_{j=\ell+1}^d \frac{ L_n^{(j)} }{ L_n^{(1)} }  \geq 1
\end{equation}
for all large enough $n\geq 1$.

For all large enough $n\geq 1$ so that the above inequalities \eqref{eq211}, \eqref{eq212}, and \eqref{eq22} hold, we have
\begin{alignat*}{3}
P_n & (L_n^{(1)})^t \rho_{n+1}^t \prod_{i=n+1}^{D_n} \rho_i^\ell C_i \geq P_n^{1-\eps} (L_n^{(1)})^t \rho_{n+1}^t \prod_{i=n+1}^{D_n+1} \rho_i^\ell C_i  && \text{by } \eqref{eq211}\\
& = P_n^{1-\eps} (L_n^{(1)})^t \rho_{n+1}^t \left( \prod_{j=\ell+1}^d \frac{ L_n^{(j)} }{ L_{D_n+1}^{(j)} } \right) \cdot \prod_{i=n+1}^{D_n+1} \left( \rho_i^\ell C_i \prod_{j=\ell+1}^d \frac{ L_i^{(j)} }{ L_{i-1}^{(j)} } \right) \\
& \geq P_n^{1-\eps} (L_n^{(1)})^t \rho_{n+1}^t \left( \prod_{j=\ell+1}^d \frac{ L_n^{(j)} }{ L_{n}^{(1)} } \right) \cdot \prod_{i=n+1}^{D_n+1} \left( \rho_i^\ell C_i \prod_{j=\ell+1}^d \frac{ L_i^{(j)} }{ L_{i-1}^{(j)} } \right)   \\
& \geq P_n^{1-\eps} (L_n^{(1)})^t \left( \prod_{j=\ell+1}^d \frac{ L_n^{(j)} }{ L_{n}^{(1)} } \right) P_n^{-\eps}  \quad && \text{by } \eqref{eq212} \\
& \geq P_n^{1-\eps} (L_n^{(1)})^t \left( \prod_{j=\ell+1}^d \frac{ L_n^{(j)} }{ L_{n}^{(1)} } \right)^{1-3\eps} P_n^{-2\eps} P_n^\eps \\
& = \left( P_n (L_n^{(1)})^{t/(1-3\eps)} \prod_{j=\ell+1}^d \frac{ L_n^{(j)} }{ L_{n}^{(1)} } \right)^{1-3\eps} P_n^\eps \\
& \geq P_n^\eps \quad && \text{by } \eqref{eq22}.
\end{alignat*}
It follows that for all large enough $n\geq 1$,
\[
\log \left( P_n (L_n^{(1)})^t \rho_{n+1}^t \prod_{i=n+1}^{D_n} \rho_i^\ell C_i \right)
\geq \eps \log P_n
\gg \eps n^2
\geq \eps \frac{ n }{k-1} (D_n-n).
\]
Here, the implied constant for $\gg$ is independent of $n$.
Hence $\dim \cF(\cT,\beta) \geq t$ by Theorem \ref{FHThm1}.
Since we choose arbitrary $t$ with $d-\ell<t<s$, we conclude Corollary \ref{FHCor24}.
\end{proof}

By \textit{elementary squares} of $\beta(\tau)$ for $\tau \in \cT$, we mean closed squares contained in $\beta(\tau)$ whose side length is equal to $L^{(1)}(\tau)$.

\begin{lem}\label{FHLem25}
For $n \in \bN \cup \{ 0 \}$ with $D_n > n$, let $\kappa \in \cT_n$ and $\tau \in \cT_{i-1}$ where $n+1 \leq i \leq D_n$.
Then for any elementary square $S$ of $\beta(\kappa)$,
\[
\# \{ \tau' \in \cT(\tau) : \beta(\tau') \cap S \neq \varnothing \}
\leq (16d)^d \rho_i^{-\ell}.
\]
\end{lem}

\begin{proof}
Through this proof, we denote the size of a rectange $R$ in $\bR^d$ by $l_1(R) \times \cdots \times l_d(R)$.

For a fixed elementary square $S$ of $\beta(\kappa)$, let $R_0 = \beta(\tau) \cap S$ and
\[
\cS = \{ \beta(\tau') \cap S : \tau' \in \cT(\tau), \beta(\tau') \cap S \neq \varnothing \}.
\]
If $R_0 = \varnothing$, then there is nothing to prove since $\# \{ \tau' \in \cT(\tau) : \beta(\tau') \cap S \neq \varnothing \} = 0$.

%Thus we assume $R_0 \neq \phi$.
%Then
%\begin{align*}
%l_j(R_0) = \min \{ L_{i-1}^{(1)}, L_n^{(1)} \} = L_{i-1}^{(1)} \quad \text{for } j=1,\dots,\ell \\
%l_j(R_0) \leq \min \{ L_{i-1}^{(j)}, L_n^{(1)} \} \leq L_n^{(1)} \quad \text{for } j=\ell+1, \dots, d.
%\end{align*}
%In the same way, for $R \in \cS$,
%\begin{align*}
%l_j(R) = L_i^{(1)}  \quad \text{for } j=1,\dots,\ell \\
%l_j(R) \leq L_n^{(1)} \quad \text{for } j=\ell+1,\dots,d.
%\end{align*}

Let $j_i \in \{1,\dots,d-1\}$ be an integer such that $L_i^{(d)} \geq \cdots \geq L_i^{(j_i+1)} \geq L_n^{(1)} > L_i^{(j_i)} \geq \cdots \geq L_i^{(1)}$.
Note that $j_i \geq \ell$.
Let $R_0'$ be the rectangle with the same center as $R_0$ such that
\[
l_j(R_0') = 
\begin{cases}
4 L_{i-1}^{(j)} & \quad \text{for } j=1,\dots,j_i \\
4 L_n^{(1)} & \quad \text{for } j=j_i+1,\dots,d.
\end{cases}
\]
Similarly, for $R\ \in \cS$ let $R'$ be the rectange with the same center such that
\[
l_j(R') = 
\begin{cases}
L_i^{(j)} + \frac{\rho_i}{4 \sqrt{d}} L_{i-1}^{(1)} & \quad \text{for } j=1,\dots,j_i \\
L_n^{(1)} & \quad \text{for } j=j_i+1,\dots,d.
\end{cases}
\]

We denote by  $z_0$ (resp. $z$) the center of $R_0'$ (resp. $R'$).
Here, we note that $z_0$ and $z$ are contained in both $\beta(\tau)$ and $S$.
For $x \in R'$ and $j=1,\dots,d$,
\[
|x_j - (z_0)_j | \leq |x_j - z_j | + | r_j - (z_0)_j | \leq \frac{1}{2} l_j(R') + \min( L_{i-1}^{(j)}, L_n^{(1)} ) \leq \frac{1}{2} l_j(R_0').
\]
Thus for all $R \in \cS$, $R' \subset R_0'$.

For any distinct $R_1, R_2 \in \cS$, let $\tau_1', \tau_2' \in \cT(\tau)$ be such that $R_1 =\beta(\tau_1')\cap S$ and $R_2 = \beta(\tau_2')\cap S$, and let $z_1, z_2$ be the centers of $R_1', R_2'$, respectively.
Suppose $\| z_1 - z_2 \|_\infty = | (z_1)_j - (z_2)_j |>0$ for some $j = 1,\dots,j_i$.
Then for any $x\in R_1'$ and $y\in R_2'$, we have
\begin{align*}
| x_j - y_j | 
&\geq |(z_1)_j - (z_2)_j | - |x_j - (z_1)_j | - | y_j - (z_2)_j | \\
&\geq \frac{1}{\sqrt{d}}\mathrm{dist}(\beta(\tau_1'),\beta(\tau_2')) + L_i^{(j)} - \frac{1}{2} l_j(R_1') - \frac{1}{2} l_j(R_2') \\
&\geq \left( \frac{\rho_i}{\sqrt{d}} L_{i-1}^{(1)} + L_i^{(j)} \right) - \frac{1}{2} l_j(R_1') - \frac{1}{2} l_j(R_2') \\
&= \frac{3 \rho_i}{4 \sqrt{d}} L_{i-1}^{(1)}>0.
\end{align*}
Thus $R_1' \cap R_2' = \varnothing$.

Now we suppose $\| z_1 - z_2 \|_\infty = | (z_1)_j - (z_2)_j |>0$ for some $j = j_i+1,\dots,d$.
Observe that 
\[
L_n^{(1)} \geq l_j(R_1) + l_j(R_2) + \frac{1}{\sqrt{d}}\mathrm{dist}(\beta(\tau_1'),\beta(\tau_2')) > l_j(R_1) + l_j(R_2),
\]
which implies that
\[
|(z_1)_j - (z_2)_j | = L_n^{(1)} -\frac{1}{2}l_j(R_1) -\frac{1}{2}l_j(R_2) > \frac{1}{2}L_n^{(1)}.
\]
Thus, for any fixed $R_1 \in \cS$ and $j = j_i+1, \dots, d$,
$$ \#\{ R_2 \in \cS\setminus\{R_1\} : \| z_1 - z_2 \|_\infty = | (z_1)_j - (z_2)_j | \text{ and } R_1' \cap R_2' \neq \varnothing \} \leq 1.$$ 

Combining the above two arguments, we conclude that every points of $R_0'$ is covered by at most $d-j_i +1$ rectangles of $\{R' : R \in \cS \}$.
It follows that
\begin{align*}
\left( \frac{\rho_i}{4\sqrt{d}} L_{i-1}^{(1)} \right)^{j_i} \left( L_n^{(1)} \right)^{d-j_i} \# \cS 
& \leq \left( L_i^{(j)} + \frac{\rho_i}{4\sqrt{d}} L_{i-1}^{(1)} \right)^{j_i} \left( L_n^{(1)} \right)^{d-j_i} \# \cS \\
& = \vol (R') \# \cS \\
& \leq (d-j_i+1) \vol (R_0') \\
& \leq d 4^d \left( L_{i-1}^{(1)} \right)^{j_i} \left( L_n^{(1)} \right)^{d-j_i},
\end{align*}
hence, using $j_i \geq \ell$,
\[
\# \cS \leq  d^{1+j_i/2} 4^{d+j_i} \rho_i^{-j_i} \leq (16d)^d \rho_i^{-\ell}. 
\]
This inequality completes the proof.
\end{proof}

%
%Since the any two rectanges in $\cS$ are separated by a distance at least $\rho_i L_{i-1}^{(1)}$, every points of $R_0'$ is covered by at most $2^{d-\ell}$ rectangles of $\{R' : R \in \cS \}$.
%For any $R \in \cS$, $R'$ is contaned in $R_0'$ by contruction.
%Thus
%\begin{align*}
%\left( \frac{\rho_i}{4} L_{i-1}^{(1)} \right)^\ell \left( L_n^{(1)} \right)^{d-\ell} \# \cS 
%& \leq \left( L_i^{(1)} + \frac{\rho_i}{4} L_{i-1}^{(1)} \right)^\ell \left( L_n^{(1)} \right)^{d-\ell} \# \cS \\
%& = \vol (R') \# \cS \\
%& \leq 2^{d-\ell} \vol (R_0') \\
%& = 2^{d-\ell} 3^d \left( L_{i-1}^{(1)} \right)^\ell \left( L_n^{(1)} \right)^{d-\ell}.
%\end{align*}
%This inequality completes the proof.
%\end{proof}
Let $\mu$ be the unique probability measure on $\cF(\cT,\beta)$ satisfying the following property: 
For all $y \in \cF(\cT,\beta)$ and $n \in \bN$,
\eqlabel{msrprop}{ \frac{ \mu ( \{ x \in \cF(\cT,\beta) : \tau_n(x) = \tau_n(y) \} ) }{ \mu ( \{ x \in \cF(\cT,\beta) : \tau_{n-1}(x) = \tau_{n-1}(y) \} ) } 
= \frac{ 1 }{ \# \cT(\tau_{n-1}(y)) }, }
where $x = \bigcap_{n\geq0} \beta(\tau_n(x))$.
We remark that for any $n\in\bN$ and $\kappa \in \cT_n$, it follows from \eqref{msrprop} that
\eqlabel{trivialprop}{
\mu(\beta(\kappa)) \leq \frac{\mu(\cF(\cT,\beta))}{C_0 \dots C_n} =\frac{1}{P_n}.
}
\begin{lem}\label{FHLem26}
Let $n \in \bN$ and $\kappa \in \cT_n$.
Then for any elementary square $S$ of $\beta(\kappa)$, one has
\[
\mu(S) \leq (16d)^{d(D_n-n)} P_n^{-1} \prod_{i=n+1}^{D_n} \rho_i^{-\ell} C_i^{-1}.
\]
\end{lem}

\begin{proof}
If $D_n=n$, then it follows from \eqref{trivialprop}.
Assume $D_n > n$.
Applying Lemma \ref{FHLem25} for $i=n+1,\dots,D_n$, we have
\begin{equation}\label{eq24}
\# \{ \tau \in \cT_{D_n} : \beta(\tau) \cap S \neq \varnothing \}
\leq (16d)^{d(D_n-n)} \prod_{i=n+1}^{D_n} \rho_i^{-\ell}.
\end{equation}
Since $S \cap \cF(\cT,\beta)$ can be covered by rectangles $\{ \beta(\tau) : \tau \in \cT_{D_n}, \beta(\tau) \cap S \neq \varnothing \}$, we have
\begin{align*}
\mu(S)
& \leq \sum_{\substack{\tau \in \cT_{D_n} \\ \beta(\tau) \cap S \neq \varnothing}} \mu( \beta(\tau) ) \\
& \leq \mu( \beta(\kappa) ) \prod_{i=n+1}^{D_n} C_i^{-1} \cdot \# \{ \tau \in \cT_{D_n} : \beta(\tau) \cap S \neq \varnothing\} \\
& \leq (16d)^{d(D_n-n)} P_n^{-1} \prod_{i=n+1}^{D_n} \rho_i^{-\ell} C_i^{-1}.
\end{align*}
In the last inequality, we use \eqref{trivialprop} and \eqref{eq24}.
\end{proof}

Let $U$ be an open subset of $\bR^d$ with $U \cap \cF(\cT,\beta) \neq \varnothing$.
If $U \cap \cF(\cT,\beta)$ is a single point set, then we denote by $n(U)$ the smallest $n \in \bN$ such that $\diam(U) \geq \rho_{n+1} L_n^{(1)}$. 
In that case, there is a unique $\kappa = \kappa(U) \in \cT_{n(U)}$ such that $U \cap \cF(\cT,\beta) \subset \beta(\kappa)$.
If $U \cap \cF(\cT,\beta)$ contains more than two points, then we denote by $n(U)$ the largest $n \in \bN$ such that $U \cap \cF(\cT,\beta) \subset \beta(\kappa)$ for some $\kappa = \kappa(U) \in \cT_n$.
We note that $\diam(U) \geq \rho_{n(U)+1} L_{n(U)}^{(1)}$ by the assumption \eqref{thm2pro4} of Theorem \ref{FHThm1}.

\begin{lem}\label{FHLem27}
Let $U$ be an open subset of $\bR^d$ with $U \cap \cF(\cT,\beta) \neq \varnothing$.
Let $n = n(U)$ and $\kappa = \kappa(U)$.
Then there is a family $\cS$ of elementary squares of $\beta(\kappa)$ such that
\begin{enumerate}
\item $\bigcup_{S \in \cS} S \supset U \cap \cF(\cT,\beta)$;
\item $\left( L_n^{(1)} \right)^t \cdot \#\cS \leq 2^{d-\ell} \rho_{n+1}^{-t} \diam (U)^t$ for all $t \geq d-\ell$.
\end{enumerate}
\end{lem}

\begin{proof}
If $\diam (U) \leq L_n^{(1)}$, then there exists an elementary square $S$ of $\beta(\kappa)$ such that $S \supset U \cap \cF(\cT,\beta)$.
We set $\cS = \{ S \}$. Then we can verify that the two assertions of the lemma hold. 

Now we assume $\diam (U) > L_n^{(1)}$.
Then $U \cap \cF(\cT,\beta)$ can be covered by $\left\lceil \frac{\diam(U)}{L_n^{(1)}} \right\rceil^{d-\ell}$ elementary sqaures.
Let $\cS$ be the family of these elementary squares.
Then it follows from $\rho_{n+1}\leq 1$ that
\begin{align*}
\left( L_n^{(1)} \right)^t \cdot \# \cS
& = \left( L_n^{(1)} \right)^t \left\lceil \frac{\diam(U)}{L_n^{(1)}} \right\rceil^{d-\ell}
\leq 2^{d-\ell} \left( \frac{\diam(U)}{L_n^{(1)}} \right)^{d-\ell} \left( L_n^{(1)} \right)^t \\
& \leq 2^{d-\ell} \left( \frac{\diam(U)}{L_n^{(1)}} \right)^t \left( L_n^{(1)} \right)^t
\leq 2^{d-\ell} \rho_{n+1}^{-t} \diam(U)^t.
\end{align*}
\end{proof}

\begin{proof}[Proof of Theorem \ref{FHThm1}]
For a real number $t$ such that $d-\ell \leq t < s$, there exists $n_0 = n_0(t)$ such that for all $n \geq n_0$,
\begin{equation}\label{eq26}
P_n \left( L_n^{(1)} \right)^t \rho_{n+1}^t \prod_{i=n+1}^{D_n} \rho_i^\ell C_i
\geq ( 16d )^{d \max \{ D_n-n, 1 \}}
\geq ( 16d )^{d(D_n-n)}.
\end{equation}

Let $\cU$ be an open cover of $\cF (\cT,\beta)$.
Assume that for all $U \in \cU$, $\diam (U)$ is small enough so that $n(U) \geq n_0$.
Since $\cF (\cT,\beta)$ is compact, there exists a finite subcover $\cU_0$ such that for all $U \in \cU_0$, $U \cap \cF (\cT,\beta) \neq \varnothing$.

For $U \in \cU_0$, let $\cS_U$ be a family of elementary squares given by Lemma \ref{FHLem27}.
Let $\cQ = \bigcup_{U \in \cU_0} \cS_U$ and $n(S) = n(U)$ for $S \in \cS_U$.
We note that $S$ may belong to different $\cS_U$.
However, $n(S)$ is well-difined since a side length of $S$ is $L_{n(U)}^{(1)}$.
Then $\cQ$ covers $\cF (\cT,\beta)$ and hence
\begin{alignat*}{3}
\sum_{U \in \cU_0} \diam (U)^t
& \geq \frac{1}{2^{d-\ell}} \sum_{S \in \cQ} \rho_{n(S)+1}^t \left( L_{n(S)}^{(1)} \right)^t \quad &&\text{by Lemma } \ref{FHLem27} \\
& \geq \frac{1}{2^{d-\ell}} \sum_{S \in \cQ} ( 16d )^{d(D_n-n)} P_{n(S)}^{-1} \prod_{i=n+1}^{D_n} \rho_i^{-\ell} C_i^{-1} \quad  &&\text{by } \eqref{eq26} \\
& \geq \frac{1}{2^{d-\ell}} \sum_{S \in \cQ} \mu(S) \quad &&\text{by Lemma } \ref{FHLem26} \\
& \geq \frac{1}{2^{d-\ell}}.
\end{alignat*}
Thus we have $\dim \cF(\cT,\beta) \geq t$.
Since we choose arbitrary $t$ with $d-\ell \leq t < s$, the proof is completed.
\end{proof}

% Section : Counting lattice points in convex sets %%%%%%%%%%%%%%%%%%%%%%%%%%%%%%%%%%%%%%%%%%%%%%%%%%%%%%%%%%%%

\section{Counting lattice points in convex sets}\label{sec3}
In this section, we will generalize the results in \cite[\textsection 3.2]{LSST} for $\bR^3$ to $\bR^{d+1}$ $(d\geq 3)$. 
We first recall the notations and lemmas in \cite[\textsection 3.1]{LSST}.

% subsection : Lattice point counting in $\mathbb{R}^d$ %%%%%%%%%%

\subsection{Preliminaries for lattice point counting}\label{subsec3.1}
For a positive integer $D\geq 1$, we write the $D$-dimensional Euclidean space by $\cE_D =\bR^D$ .
For a convex body $K \subset \bR^D$ and a lattice $\Lambda\subset\bR^D$, let $\lambda_{i}(K,\Lambda)$ $(i=1,\dots,D)$ be the $i$-th successive minimum of $\Lambda$ with respect to $K$, that is, the infimum of those numbers $\lambda$ such that $\lambda K \cap \Lambda$ contains $i$ linearly 
independent vectors. Let $\vol(\cdot)$ be the Lebesgue measure on $\bR^D$ and let $\cov(\Lambda)$ be the covolume of a lattice $\Lambda$, which
is the Lebesgue measure of a fundamental domain of $\Lambda$. Denote 
\[
\theta(K,\Lambda):=\frac{\vol(K)}{\cov(\Lambda)}.
\]
For an affine subspace $H$ of $\bR^D$, let $\vol_H(\cdot)$ be the Lebesgue measure on $H$ with respect to the subspace Riemannian structure.
We write $\vol_H(S)=\vol_H(S\cap H)$ for a Borel measurable subset $S$ of $\bR^D$ by abuse of notation.
We say that a subspace $H$ of $\bR^D$ is $\Lambda$-rational if $H\cap \Lambda$ is a lattice in $H$, 
and denote by $\cov_H (\Lambda)$ the covolume of the lattice $H\cap \Lambda$ in $H$. 
We also use the same notations for the dual vector space $\cE_D^*$ with respect to the standard Euclidean structure.
For any $\vphi\in \cE_D^*$, denote $H_{\vphi}=\mathrm{ker}\vphi$.

We use $\|\cdot \|$ for the Euclidean norms on $\bR^D$ and $\cE_D^*$. For a normed vector space $V$, 
denote by $B_r (V)$ (or $B_r$ if $V=\bR^D$) the ball of radius $r$ centered at $0\in V$. We use $K$-norms on 
$\bR^D$ and $\cE_D^*$ defined by 
\[
\begin{cases}
\|\mb{v}\|_K = \inf \{r>0 : \mb{v}\in rK\}, \quad &\mb{v}\in\bR^D,\\
\|\vphi \|_K = \sup_{\mb{v}\in K}|\vphi(\mb{v})|, \quad &\vphi\in\cE_D^*.
\end{cases}
\] 

Recall that $\cL_D$ is the space of unimodular lattices in $\bR^D$, which can be identified with the homogeneous space $\SL_D(\bR)/\SL_D(\bZ)$. For $g\in\SL_D(\bR)$ let $g^*$ be the adjoint action on $\cE_D^*$ defined by
$\vphi\mapsto \vphi \circ g$. Then $g^*$ can be represented by the transpose of $g$ with respect to the standard
basis $\mb{e}_1,\dots,\mb{e}_D$ of $\bR^D$ and the dual basis $\mb{e}_1^*,\dots,\mb{e}_D^*$ of $\cE_D^*$. 

The \textit{dual lattice} of $\Lambda$ in $\bR^D$ is the lattice in $\cE_D^*$ defined by
\[
\Lambda^* = \{\vphi\in\cE_D^*: \vphi(\mb{v})\in\bZ,\ \forall\mb{v}\in\Lambda\}.
\]
Let us define the following two sets:
\[
\begin{split}
\cK_{\eps}&=\cK_{\eps}(D)=\{\Lambda\in\cL_D : \|\mb{v}\|\geq \eps,\ \forall \mb{v}\in\Lambda\smallsetminus\{0\}\}=\{\Lambda\in\cL_D:\lambda_1 (B_1,\Lambda)\geq \eps\};\\
\cK_{\eps}^* &=\cK_{\eps}^* (D)=\{\Lambda\in\cL_D : \|\vphi\|\geq \eps,\ \forall\vphi\in\Lambda^*\smallsetminus\{0\}\}.
\end{split}
\]
Since $\cE_D^*$ can be naturally identified with $\bigwedge_{\bR}^{D-1}\bR^D$ with the standard Euclidean
structure, we have $\Lambda^* = \bigwedge_{\bZ}^{D-1}\Lambda$. 

We define the subset $\cL_D'$ of $\cL_D$ by
\[ \cL_D' = \{ \Lam \in \cL_D : \Lam \cap \bR \be_D = r \bZ \be_D \text{ for some $r$ with } 1/2 < r \leq 1 \}, \]
which plays an important role to prove the main theorem.

A nonzero vector $\mb{v}\in\Lambda$ is said to be \textit{primitive} if $(1/n)\mb{v}\notin \Lambda$ for all $n\in\bN$. The set of primitive vectors in $\Lambda$ is denoted by $\wh{\Lambda}$.

We summarize the lemmas in \cite[\textsection 3.1]{LSST}.
\begin{lem}\label{CLPCSLem1}
Let $D \geq 2$. For every lattice $\Lambda$ in $\mathbb{R}^D$ and every bounded centrally symmetric convex subset $K$ of $\mathbb{R}^D$ with $\lambda_d(K,\Lambda) \leq 1$ we have
\eq{
\# ( K \cap \widehat{\Lambda} ) = \left( \zeta(D)^{-1} + \eta(K,\Lambda) \right) \cdot \theta(K,\Lambda)
}
where $\zeta$ is the Riemann $\zeta$-function and
\eq{
| \eta(K,\Lambda) | \ll_D \lambda_D(K,\Lambda) - \lambda_D(K,\Lambda) \log \lambda_1(K,\Lambda).
}
\end{lem}

\begin{lem}\label{CLPCSLem2}
Let $D \geq 2$. For every lattice $\Lambda$ in $\mathbb{R}^D$ and every bounded centrally symmetric convex subset $K$ of $\mathbb{R}^D$ with $\lambda_D(K,\Lambda) \leq 1$ we have
\eq{
\# (K \cap (\Lambda \smallsetminus \{0\})) = (1+\alpha(K,\Lambda)) \cdot \theta(K,\Lambda)
}
where $|\alpha(K,\Lambda)| \ll_D \lambda_D(K,\Lambda)$
\end{lem}

\begin{lem}\label{CLPCSLem3}
Let $K$ and $\Lambda$ be as in Lemma \ref{CLPCSLem2}. Then
\eq{
\# (K \cap \Lambda) \asymp_D \theta(K,\Lambda).
}
\end{lem}

\begin{lem}\label{CLPCSLem4}
Let $D \geq 1$. Let $\Lambda$ be a lattice in $\mathbb{R}^D$ and $K$ be a bounded centrally symmetric convex subset of $\mathbb{R}^D$ with nonempty interior. Then
\eq{
\# (K^\circ \cap \Lambda) \asymp_D \# (K \cap \Lambda) \asymp_D \# (\overline{K} \cap \Lambda).
}
\end{lem}

\begin{lem}\label{CLPCSLem5}
Let $K$ and $\Lambda$ be as in Lemma \ref{CLPCSLem2}. If $\lambda_i(K,\Lambda) \leq s \leq s' \leq \lambda_{j+1} (K, \Lambda)$ where $ 1 \leq i \leq j \leq D$, then
\eq{
\left( \frac{s'}{s} \right)^i 
\ll_D \frac{ \# (s'K \cap \Lambda) }{ \# (sK \cap \Lambda) } 
\ll_D \left( \frac{s'}{s} \right)^j.
}
\end{lem}

\begin{lem}\label{CLPCSLem6}
Let $D \geq 2$. Let $K$ be a bounded centrally symmetric convex subset of $\mathbb{R}^D$ with nonempty interior and let $\varphi \in \cE_D^* \smallsetminus \{0\}$. Then
\eq{
\mathrm{vol}_{H_\varphi} (K) \asymp_D \| \varphi \| \mathrm{vol}(K) / \| \varphi \|_K.
}
\end{lem}

We need the following auxiliary lemma.
\begin{lem}\label{LBLem3}
Given $D\geq 2$ and $r>0$, let $\Lambda \in \cK_r^*(D)$, and let $\bv, \bw \in \Lambda$ be any nonzero linearly independent vectors. Then there exists a positive constant $c'=c'(D)>0$ such that $\| \bv \wedge \bw \| \geq c' r^{D-2}$.
\end{lem}

\begin{proof}
Let $\Lambda'$ be the 2-dimensional sublattice of $\Lambda$ generated by $\bv, \bw$. By Minkowski's second theorem, we have
\eqlabel{mst1}{
\| \bv \wedge \bw \| \geq \text{cov}(\Lambda') \gg_2 \lambda_1(B_1, \Lambda') \lambda_2(B_1,\Lambda') 
\geq \lambda_1(B_1,\Lambda) \lambda_2(B_1,\Lambda). 
}
Again by Minkowski's second theorem, we have
\eqlabel{mst2}{
\begin{split}
1 \ll_D \lambda_1(B_1,\Lambda) \cdots \lambda_{D}(B_1,\Lambda) &\leq \lambda_1(B_1,\Lambda) \lambda_2(B_1,\Lambda) \lambda_{D}(B_1,\Lambda)^{D-2} \\
&\leq \lambda_1(B_1,\Lambda) \lambda_2(B_1,\Lambda) \frac{1}{r^{D-2}}.
\end{split}
}
The last inequality comes from $\Lambda \in \cK_r^*(D)$. We obtain the result by combining (\ref{mst1}) and (\ref{mst2}).
\end{proof}

\subsection{Lattice point counting in $\mathbb{R}^{d+1}$}\label{subsec3.2}

For $d \geq 2$ and a $(d+1)$-tuple $\bbr = (r_1, \dots , r_{d+1})$ of positive real numbers, we estimate the number of lattice points in the set
\eq{
M_\bbr = \{ (x_1, \dots , x_{d+1}) \in \bR^{d+1} : |x_i| \leq r_i,\ \forall i=1,\dots,d+1 \}.
}
Let
\eq{
M_\bbr^* = \{ \varphi \in \cE_{d+1}^* : | x_i^\varphi | \leq r_i,\ \forall i = 1, \dots , d+1 \},
}
where the element $\varphi \in \cE_{d+1}^*$ is represented by $\varphi = \sum_{i=1}^{d+1} x_i^\varphi \be_i^*$.

\begin{lem}\label{CLPCSLem7}
Let $d \geq 2$.
For any real number $c_0 >1$, there exists a positive real number $\tilde{c} < 1$ such that for every lattice $\Lam$ in $\bR^{d+1}$ and every $(d+1)$-tuple $\bbr$ of positive real numbers with 
\eq{
\lam_{d+1}(M_\bbr,\Lam) \leq \tilde{c} \quad \text{and} \quad -\lam_{d+1}(M_\bbr,\Lam) \log \lam_1(M_\bbr,\Lam) \leq \tilde{c}
}
one has
\eq{
\frac{ 1 }{ c_0 \zeta(d+1) } \theta(M_\bbr,\Lam)
\leq \# (M_\bbr \cap \widehat{\Lam})
\leq \frac{ c_0 }{ \zeta(d+1) } \theta(M_\bbr,\Lam).
}
\end{lem}

\begin{proof}
The proof follows directly from Lemma \ref{CLPCSLem1}.
\end{proof}

Now we fix real numbers $s, r_1, \dots , r_{d+1}$ such that $0 < s < 1/2$, 
$r_i \geq 1$ for each $i=1,\dots,d$, and $r_{d+1}=1$. 
Denote by $\mb{r} = (r_1,\dots,r_{d+1})$, $r_{M}=\max_{1\leq i \leq d} r_i$, and $r_m = \min_{1\leq i \leq d} r_i$. Define a norm
\eq{
\| \varphi \|_\bbr = \max \left\{ r_i |x_i^\varphi| : i = 1,\dots,d+1 \right\}.                                                                                                                                                                                                                                                                                                                                                                                                                                                                                                                                                                                                                                                                                                                                                                                                                                                                                                                                                                                                                                                                                                                                                                                                                                                                                                                                                                                                                                                                                                                                                                                                                                                                                                                                                                                                                                                                                                                                                                                                                                                                                                                                                                                                                                                                                                                                                                                                                                                                                                                                                                                                                                                                                                                                                                                                                                                                                                                                                                                                                                                                                                                                                                                                                                                                                                                                                                                                                                                                                                                                                                                                                                                                                                                                                                                                                                                                                                                                                                                                                                                                                                                         
}
It follows from the definition that
\eqlabel{normbound}{
\| \varphi \|_\bbr 
\leq \| \varphi \|_{M_\bbr} 
\leq (d+1) \| \varphi \|_\bbr.
}

For $q>0$ let 
\eq{
N_q(\bbr,s) = \left\{ \varphi \in \cE_{d+1}^* : |x_i^\varphi| \leq s,\ \forall i = 1,\dots,d, \text{ and } \| \varphi \|_\bbr \leq q \right\}.
}
Note that $N_q (\bbr,s) = M^*_{\bbr'}$ where $\bbr' = (r_1',\dots,r_{d}',q)$ with $r_i' = \min \{ q/r_i, s \}$.
For a lattice $\Lam$ in $\bR^{d+1}$ and $i = 1,\dots,d+1$, let $q_i(\Lam,\bbr,s)$ be the infimum of those positive real numbers $q$ such that $N_q(\bbr,s) \cap \Lam^*$ contains $i$ linearly independent vectors.
We will give an upper bound of the number of
\eq{
\cS(\Lam,\bbr,s) := \left\{ \bv \in M_\bbr \cap \widehat{\Lam} : \varphi(\bv) = 0 \text{ for some } \varphi \in N_{(d+1)sr_M}(\bbr,s) \cap \widehat{\Lam^*} \right\},
}
where $\widehat{\Lam^*}$ is the set of primitive vectors in $\Lam^*$.

\begin{lem}\label{CLPCSLem8}
For $d \geq 2$, let $\Lam$ be a unimodular lattice in $\bR^{d+1}$ with $q_1(\Lam,\bbr,s) \geq s^{-2}$.
Then
\begin{enumerate}
\item\label{count_lem_1} if $r_m = r_M$ and $q_{d+1}(\Lam,\bbr,s) \leq d s^{-1/2} r_M$, then $$\# \cS(\Lam,\bbr,s) \ll s^{1/2} \cdot \mathrm{vol}(M_\bbr);$$
\item\label{count_lem_2} if $r_m < r_M$ and $q_{d+1}(\Lam,\bbr,s) \log q_{d+1}(\Lam,\bbr,s) \leq sr_M$, then $$\# \cS (\Lam,\bbr,s) \ll s^2 \cdot \vol(M_\bbr).$$
\end{enumerate}
\end{lem}

\begin{proof}
For simplicity, we denote by $N_q = N_q(\bbr,s)$, $q_i = q_i(\Lam,\bbr,s)$ and $\cS = \cS(\Lam,\bbr,s)$.
If $N_{(d+1)sr_M} \cap \widehat{\Lam^*}$ is empty then there is nothing to prove.
We assume that $N_{(d+1)sr_M} \cap \widehat{\Lam^*}$ is nonempty.
It follows from the definition that
\eqlabel{ubS1}{
\# \cS \leq \sum_{\varphi \in N_{(d+1)sr_M} \cap \widehat{\Lam^*}} \# ( H_\varphi \cap M_\bbr \cap \widehat{\Lam} ).
}

We first claim that for every $\varphi \in N_{(d+1)sr_M} \cap \widehat{\Lam^*}$,
\eqlabel{ubS2}{
\# (H_\varphi \cap M_\bbr \cap \widehat{\Lam})                                                                                                                                                                                                                                                                                                                                                                                                                                                                                                                                                                                                                                                                                                                                                                                                                                                                                                                                                                                                                                                                                                                                                                                                              
\ll \frac{ \mathrm{vol}(M_\bbr) }{ \|\varphi\|_{M_\bbr} } 
\leq \frac{ \vol(M_\bbr) }{ \|\varphi\|_\bbr },
}
where the second inequality follows from \eqref{normbound}.
In fact, if $\# (H_\varphi \cap M_\bbr \cap \widehat{\Lam}) < d+1$, then it follows from \eqref{normbound} that
\eq{
\frac{ \mathrm{vol}(M_\bbr) }{ \|\varphi\|_{M_\bbr} }\geq \frac{ 2^{(d+1)} r_1 \dots r_{d+1}}{ (d+1)\|\varphi\|_{\mb{r}} } \geq \frac{ 2^{(d+1)} r_1 \dots r_{d+1}}{ (d+1)^2 sr_M } \gg \# (H_\varphi \cap M_\bbr \cap \widehat{\Lam}).
}
Otherwise, $H_\varphi \cap M_\bbr \cap \Lam$ has $d$ linearly independent vectors, hence it follows from Lemma \ref{CLPCSLem3} and Lemma \ref{CLPCSLem6} that
\eq{
\# (H_\varphi \cap M_\bbr \cap \widehat{\Lam})
\ll \frac{ \vol_{H_\varphi}(M_\bbr) }{ \cov_{H_\varphi}(\Lam) }
\ll \frac{ \| \varphi \| \vol(M_\bbr) }{ \cov_{H_\varphi}(\Lam) \| \varphi \|_{M_\bbr} }
\ll \frac{ \vol(M_\bbr) }{ \| \varphi \|_{M_\bbr} },
}
which concludes the claim.

By \eqref{ubS1} and \eqref{ubS2}, to prove Lemma \ref{CLPCSLem8}, it suffices to estimate
\eqlabel{ubS3}{
\begin{split}
\eta :=& \sum_{\varphi \in N_{(d+1)sr_M} \cap \widehat{\Lam^*}} \| \varphi \|_\bbr^{-1}  \\
=& \frac{ 1 }{ (d+1)sr_M } \# (N_{(d+1)sr_M} \cap \widehat{\Lam^*} ) + \sum_{\varphi \in N_{(d+1)sr_M} \cap \widehat{\Lam^*}} \int_{\|\varphi\|_\bbr}^{(d+1)sr_M} \frac{ 1 }{ q^2 } \mathrm{d}q. 
\end{split}
}
We denote the first and second terms in the last line by $\eta_1,\eta_2$, respectively.

Observe that
\eqlabel{ubS8}{
\begin{split}
\eta_2 
&= \sum_{\varphi \in N_{(d+1)sr_M} \cap \widehat{\Lam^*}} \int_{q_1}^{(d+1)sr_M} \frac{ \mathbbm{1}_{q}(\|\varphi\|_\bbr) }{ q^2 } \mathrm{d}q \\
&= \int_{q_1}^{(d+1)sr_M} \sum_{\varphi \in N_{(d+1)sr_M} \cap \widehat{\Lam^*}} \frac{ \mathbbm{1}_{q}(\|\varphi\|_\bbr) }{ q^2 } \mathrm{d}q  \\
&\leq \int_{q_1}^{(d+1)sr_M} \frac{ \# (N_q \cap \widehat{\Lam^*}) }{ q^2 } \mathrm{d}q. 
\end{split}
}
where $\mathbbm{1}_q$ denotes the indicator function of the set $\{x \in \bR : x\leq q\}$.
%In the second equality we use Fubini's theorem.

On the other hand, for $i = 2,\dots,d$, if $q_{i-1} \leq q < q_i$ then we have $\# (N_q \cap \widehat{\Lam^*}) = i \leq d$.
Thus
\eqlabel{ubS5}{
\int_{q_1}^{q_{d}} \frac{ \# (N_q \cap \widehat{\Lam^*}) }{ q^2 } \mathrm{d}q 
\leq \int_{q_1}^{q_{d}} \frac{ d }{ q^2 } \mathrm{d}q
\leq \frac{ d }{ q_1 }
\ll s^2
\leq s^{1/2},
}
where the third inequality follows from the assumption $q_1 \geq s^{-2}$.

Now we are ready to prove the two assertions in Lemma \ref{CLPCSLem8} separately.
\begin{proof}[Proof of the assertion \eqref{count_lem_1}]
We claim that $\eta \ll s^{1/2}$ under the assumption of \eqref{count_lem_1}, which concludes the assertion \eqref{count_lem_1}. Assume that $r_m = r_M$ and $q_{d+1} \leq d s^{-1/2} r_M$.
Observe that by definition
\eqlabel{eqNM}{
N_{(d+1)s^{-1/2}r_M} = M_{(s,\dots,s,(d+1)s^{-1/2}r_M)}^*.
}
We have an upper bound of $\eta_1$ as 
\eqlabel{ubS4}{
\eta_1 
\leq \frac{ \# (N_{(d+1)s^{-1/2}r_M} \cap \Lam^*) }{ (d+1)sr_M } 
\ll \frac{ \vol(N_{(d+1)s^{-1/2}r_M}) }{ (d+1)sr_M }
\ll s^{d-3/2}
\leq s^{1/2}.
}
The first inequality follows from $s<1/2$, the second inequality follows from Lemma \ref{CLPCSLem3},
and the third inequality follows from \eqref{eqNM}.

For an upper bound of $\eta_2$, we first compute
\eqlabel{ubS6}{
\begin{split}
\int_{sr_M}^{(d+1)sr_M} \frac{ \# (N_q \cap \widehat{\Lam^*}) }{ q^2 } \mathrm{d}q
&\leq \int_{sr_M}^{(d+1)sr_M} \frac{ \# (N_{(d+1)sr_M} \cap \Lam^*) }{ q^2 } \mathrm{d}q \\
&\leq \frac{ \# (N_{(d+1)sr_M} \cap \Lam^*) }{ sr_M }  \\
&\ll s^{1/2}, 
\end{split}
}
where the last inequality can be shown in the same way as in (\ref{ubS4}).

If $sr_M \leq q_{d}$, then it follows from (\ref{ubS5}) and (\ref{ubS6}) that $\eta_2 \ll s^{1/2}$.
Now we suppose that $sr_M > q_{d}$.
For all $q_{d} < q \leq sr_M = sr_m$, observe that 
$$N_q = M_{(q/r_1,\dots,q/r_{d+1})}^* = \frac{ q }{ sr_M } N_{sr_M}.$$
%Using this identification, we have $N_q \cap \widehat{\Lam^*} = \frac{ q }{ sr_{d} } N_{sr_{d}} \cap \widehat{\Lam^*}$.
Since $\lam_{d}(N_q,\Lam)=\lam_{d} (\frac{ q }{ sr_M } N_{sr_M},\Lam) \leq 1 \leq sr_M/q$, it follows from
Lemma \ref{CLPCSLem5} that
\eq{
\# (N_q \cap \widehat{\Lam^*})
\leq \# \left( \frac{ q }{ sr_M } N_{sr_M} \cap \Lam^* \right)
\ll \left( \frac{ q }{ sr_M } \right)^{d} \# (N_{sr_M} \cap \Lam^*).
}
%where the last inequality holds by Lemma \ref{CLPCSLem5} with the fact that $\lam_{d}(N_q,\Lam)=\lam_{d} (\frac{ q }{ sr_{d} } N_{sr_{d}},\Lam^*) \leq 1 \leq sr_{d}/q$.
By $sr_M \leq ds^{-1/2}r_M$ and Lemma \ref{CLPCSLem3}, we have
\begin{align*}
\# (N_q \cap \widehat{\Lam^*})
&\ll \left( \frac{ q }{ sr_M } \right)^{d} \# (N_{ds^{-1/2}r_M} \cap \Lam^*) \\
&\ll \left( \frac{ q }{ sr_M } \right)^{d} \vol(N_{ds^{-1/2}r_M}) \\
&\ll \left( \frac{ q }{ sr_M } \right)^{2}s^{d-1/2}r_M \ll \frac{q^2 s^{-1/2}}{r_M}.
\end{align*}
The last line follows from $\frac{q}{sr_M}\leq 1$ and $s\leq 1$.
%Note that the implied constant depends only on $d$.
Thus we have
\eq{
\int_{q_{d}}^{sr_M} \frac{ \# (N_q \cap \widehat{\Lam^*}) }{ q^2 } \mathrm{d}q
\ll \int_{q_{d}}^{sr_M} \frac{s^{-1/2}}{r_M} \mathrm{d}q
\ll s^{1/2}.
}
It follows that $\eta_2 \ll s^{1/2}$ under the assumption of \eqref{count_lem_1},
which concludes the assertion \eqref{count_lem_1}.
\end{proof}

\begin{proof}[Proof of the assertion \eqref{count_lem_2}]
We will prove that $\eta \ll s^2$ under the assumption of \eqref{count_lem_2}.
Recall that we fix $0<s<1/2$. By the assumption, we have $q_{d+1} \geq q_1 \geq s^{-2} > 4$ so that $q_{d+1} < sr_M < (d+1)sr_M$ since $q_{d+1} \log q_{d+1} \leq sr_M$.
Thus $N_{(d+1)sr_M} \cap \Lam^*$ contains $d+1$ linearly independent vectors. By Lemma \ref{CLPCSLem3}, we have
\eqlabel{ubS7}{
\eta_1
\leq \frac{ \# (N_{(d+1)sr_M} \cap \Lam^*) }{ (d+1)sr_M } \ll \frac{ \vol (N_{(d+1)sr_M}) }{ (d+1)sr_M } \ll s^{d} \leq s^2.
}
By \eqref{ubS8}, it suffices to show that 
\eq{
 \int_{q_1}^{(d+1)sr_M} \frac{ \# (N_q \cap \widehat{\Lam^*}) }{ q^2 } \mathrm{d}q \ll s^2.
}
We split the domain of integration as $(q_1,q_{d}) \cup (q_{d},q_{d+1}) \cup (q_{d+1},sr_M) \cup (sr_M,(d+1)sr_M)$ and estimate the upper bounds of the integrals over these intervals.

For each $q \in (sr_M,(d+1)sr_M)$, it follows from Lemma \ref{CLPCSLem3} that $\# (N_q \cap \widehat{\Lam^*}) \ll \vol(N_q) \ll s^{d}q$. Thus we have
\eqlabel{ubS9}{
\int_{sr_M}^{(d+1)sr_M} \frac{ \# (N_q \cap \widehat{\Lam^*}) }{ q^2 } \mathrm{d}q
\ll \int_{sr_M}^{(d+1)sr_M} \frac{s^{d}}{ q } \mathrm{d}q 
= s^{d} \log (d+1) 
\ll s^2.
}
For each $q \in (q_{d+1},sr_M)$, it follows from Lemma \ref{CLPCSLem3} that $\# (N_q \cap \widehat{\Lam^*}) \ll \vol(N_q) \ll s^{d-1}q^2/r_M$. Thus we have
\eqlabel{ubS10}{
\int_{q_{d+1}}^{sr_M} \frac{ \# (N_q \cap \widehat{\Lam^*}) }{ q^2 } \mathrm{d}q
\ll \int_{q_{d+1}}^{sr_M} \frac{s^{d-1}}{ r_M } \mathrm{d}q 
%= \frac{ s^{d-1} (sr_M - q_{d+1}) }{ r_M }
\leq s^{d}
\leq s^2.
}
By (\ref{ubS5}), the integral over $(q_1,q_{d})$ is bounded above by $s^2$.

Now it remains to show that the integral over $(q_{d},q_{d+1})$ is bounded above by $s^2$.
Let $H = \Span_\bR(N_{q_{d}} \cap \Lam^*)$.
We claim that for every $q \in (q_{d},q_{d+1})$, 
\eqlabel{claimeq}{
\vol_H (N_q) \leq \frac{ q }{ q_{d+1} } \vol_H (N_{q_{d+1}}).
}
We first consider the case of $\be_{d+1}^* \notin H$. Let $\pr^*$ be the orthogonal projection onto $\Span_\bR \{ \be_1^*,\dots,\be_{d}^* \}$.
Then the volume of $\pr^* (N_q)$ is at most $q/q_{d+1}$ times the volume of $\pr^* (N_{q_{d+1}})$ since $q/r_M < q_{d+1}/r_M < s$. This implies \eqref{claimeq}.
We now assume $\be_{d+1}^* \in H$. Then it follows from $N_q \subset N_{q_{d+1}}$ that
\[\begin{split}
\vol_H (N_q) &= 2q \vol_{\pr^* (H)}\left(\pr^* (N_q)\right) \\
&\leq \frac{q}{q_{d+1}} 2q_{d+1} \vol_{\pr^* (H)}\left(\pr^* (N_{q_{d+1}})\right) =  \frac{q}{q_{d+1}} \vol_H (N_{q_{d+1}}).
\end{split}\]
%Comparing the diameters of $\be_{d+1}^*$-directional axis, the diameter of $N_{q_{d+1}}$ is $q/q_{d+1}$ times larger than the diameter of $N_q$.
%Since $N_q \subset N_{q_{d+1}}$, the claim holds in the case.
Thus the claim follows. %(\kim{should be more explained})

For each $q \in (q_{d},q_{d+1})$, we have
\begin{alignat*}{3}
\# (N_q \cap \widehat{\Lam^*}) 
& \ll \frac{ \vol_H(N_q) }{ \cov_H(\Lam^*) } &&\text{by Lemma \ref{CLPCSLem3}} \\
& \leq \frac{ q }{ q_{d+1} } \frac{ \vol_H(N_{q_{d+1}}) }{ \cov_H(\Lam^*) } &&\text{by \eqref{claimeq}} \\
& \ll \frac{ q }{ q_{d+1} } \# (N_{q_{d+1}} \cap H \cap \Lam^*) &&\text{by Lemma \ref{CLPCSLem3}} \\
& \ll \frac{ q }{ q_{d+1} } \# (N_{q_{d+1}}^\circ \cap H \cap \Lam^*) &&\text{by Lemma \ref{CLPCSLem4}} \\
& = \frac{ q }{ q_{d+1} } \# (N_{q_{d+1}}^\circ \cap \Lam^*) 
\leq \frac{ q }{ q_{d+1} } \# (N_{q_{d+1}} \cap \Lam^*) \\
& \ll \frac{ q }{ q_{d+1} } \vol(N_{q_{d+1}}) &&\text{by Lemma \ref{CLPCSLem3}} \\
& \ll s^{d-1} \frac{ q_{d+1} q }{ r_M }. 
\end{alignat*}
Therefore, we have
\eqlabel{ubS11}{
\int_{q_{d}}^{q_{d+1}} \frac{ \#(N_1 \cap \widehat{\Lam^*}) }{ q^2 } \mathrm{d}q 
\ll \int_{q_{d}}^{q_{d+1}} s^2 \frac{ q_{d+1} }{ sr_M } \frac{ 1 }{ q } \mathrm{d}q
\leq s^2 \frac{q_{d+1} \log q_{d+1} }{sr_M}
\leq s^2.
}
By combining \eqref{ubS5}, \eqref{ubS9}, \eqref{ubS10}, and \eqref{ubS11}, the proof of \eqref{count_lem_2} is completed.
\end{proof}
Combining the proofs of the two assertions, we complete the proof of Lemma \ref{CLPCSLem8}.
\end{proof}

For a weight vector $w = (w_1,\dots,w_d)$ as in the introduction, 
let $1\leq \ell \leq d-1$ be the unique integer such that 
$w_1 = \cdots = w_{\ell} > w_{\ell+1} \geq \cdots \geq w_d$,
and denote $$\xi = \max\left(1,\frac{d-\ell}{\ell}\right).$$
For a fixed lattice $\Lam \subset \bR^{d+1}$ and fixed $\mb{r},s$, we denote $q_i(\Lam,\bbr,s)$ by $q_i(\Lam)$ and $N_q(\bbr,s)$ by $N_q$ for simplicity. Let us fix a constant $C\geq 1$ which is an implied constant for the conclusion of Lemma \ref{CLPCSLem8} \eqref{count_lem_1} and \eqref{count_lem_2}.

\begin{lem}\label{CLPCSLem9}
Let $d \geq 2$, $s = \eps^2$, $\bbr = (r_1,\dots,r_{d+1}) = (\eps e^t,\dots,\eps e^t, 1)$, $\Lam \in \cK_{\eps^2}^* \cap \cL_{d+1}'$, and $a_t = \mathrm{diag}(e^{w_1 t},\dots, e^{w_d t}, e^{-t})$.
Then there exists a positive real number $\tilde{\eps} \leq 1 $ and $c=c(d) > (d+1)^{1/14}$ such that for all $\eps,t>0$ with $c e^{-w_d t/(2d^3)} < \eps < \tilde{\eps}$, one has
\[
\# \cS(a_t\Lam, \bbr, s) \leq \eps  \vol (M_\bbr).
\]
\end{lem}

\begin{proof}
We will prove the lemma for $\tilde{\eps} < 1/C^2$ and the constant $c$ will be determined later.
By Lemma \ref{CLPCSLem8} \eqref{count_lem_1}, it suffices to show that
\eqlabel{Lem9Claim}{
q_1(a_t \Lam) \geq s^{-2} \quad \mathrm{and} \quad q_{d+1}(a_t \Lam) \leq ds^{-1/2}r_d.
}
First, note that
\[
N_q \cap (a_t \Lam)^* = N_q \cap a_{-t}^* \Lam^* = a_{-t}^*(a_t^* N_q \cap \Lam^*),
\] where $a_t^{*}$ denotes the transpose of $a_t$.
Hence it is enough to show that $a_t^* N_{s^{-2}}$ has no nonzero lattice point of $\Lam^*$ for the first inequality of \eqref{Lem9Claim}.
Since $d\geq 2$ and $w_d \leq 1/d$, we have
\[
e^{-\frac{t}{7}}<e^{-\frac{w_d t}{2d^3}}< ce^{-\frac{w_d t}{2d^3}} < \eps,
\]
that is, $s^{-2} < r_1 s$.
Thus we have
\[
N_{s^{-2}} = M^*_{(s^{-2}/r_1, \dots, s^{-2}/r_1, s^{-2})}=M^*_{(\eps^{-5} e^{-t}, \dots, \eps^{-5} e^{-t}, \eps^{-4})},
\]
which implies that
\[
a_t^* N_{s^{-2}} = M^*_{(\eps^{-5} e^{(w_1-1)t},\dots, \eps^{-5} e^{(w_d-1)t}, \eps^{-4} e^{-t})}.
\]
Since for all $i=1,\dots,d$
\[
\frac{\eps}{(d+1)^{1/14}} > \frac{c}{(d+1)^{1/14}}e^{-\frac{w_d t}{2d^3}} > e^{-\frac{(d-1)w_d t}{7}} \geq e^{\frac{(w_i-1) t}{7}},
\]
we have $\eps^{-5} e^{(w_i-1)t} < \frac{\eps^2}{\sqrt{d+1}}$ for all $i=1,\dots,d$. It is clear that $\eps^{-4}e^{-t}<\eps^{-5}e^{(w_d -1)t}<\frac{\eps^2}{\sqrt{d+1}}$.
Thus $a_t^* N_{s^{-2}}$ is contained in the interior of $B_{\eps^2}(\cE_{d+1}^*)$.
Since $\Lam \in \cK_{\eps^2}^*$, there is no lattice point of $\Lam$ in $a_t^* N_{s^{-2}}$.

To show the second inequality of \eqref{Lem9Claim}, we will construct a basis for $\Lam^*$ whose vectors are contained in $N_{d s^{-1/2}r_d }=N_{de^t}$.
Since $ de^t > r_d > r_d s$, we have
\[
a_t^* N_{de^t}=a_t^* M^*_{(s,\dots,s,de^t)} = M^*_{(se^{w_1 t}, \dots se^{w_d t}, d)}.
\]
Let $1/2 < r \leq 1$ be such that $r \be_{d+1} \in \widehat{\Lam}$ from the assumption $\Lam \in \cL_{d+1}'$.
Let $\pr : \bR^{d+1} \to \bR^d$ be the orthogonal projection onto $\Span{(\be_1,\dots,\be_d)}$. 
Note that $\pr(\Lam)$ is a lattice with covolume $1/r$ in $\bR^d$.
If $\bv \in \Lam$ satisfies $\|\pr(\bv)\| = \lam_1(B_1,\pr(\Lam))$, then since $\Lambda \in \cK_{\eps^2}^*$,
it follows from Lemma \ref{LBLem3} with $D=d+1$ that
\begin{equation}\label{eq336}
\lam_1(B_1,\pr(\Lam)) \geq r \lam_1(B_1,\pr(\Lam)) = \|\bv \wedge r\be_{d+1}\| \geq \bar{c}_1 (\eps^2)^{d-1}
\end{equation}
for some $\bar{c}_1 = \bar{c}_1(d) <1$. Since $\mathrm{cov}(\pr(\Lam))=1/r$,
it follows from Minkowski's second theorem and (\ref{eq336}) that for any $0<c_1<\bar{c}_1$
\[
c_1^{d-1}(\eps^2)^{(d-1)^2}\lam_d(B_1,\pr(\Lam)) \leq \lam_1(B_1,\pr(\Lam))\cdots\lam_d(B_1,\pr(\Lam)) 
\ll 1.
\]
Hence, there exists $c_2 = c_2(d) > 1$ such that
\eqlabel{eqLam_d}{
\lam_d(B_1,\pr(\Lam)) \leq c_2 (\eps^{-2})^{(d-1)^2}.
}

Let $\{v^{(i)}: i=1,\dots,d\}$ be a Minkowski reduced basis for $\pr(\Lam)$ such that $\|v^{(i)}\| \leq 2^d \lam_i(B_1,\pr(\Lam))$. 
For each $i=1,\dots,d$, let $\bv_i \in \Lam$ be such that $\pr(\bv_i) = v^{(i)}$ and $|\be^*_{d+1}(\bv_i)|<1$.
Then the vectors $\bv_1,\dots,\bv_d,\bv_{d+1} = r \be_{d+1}$ form a basis for $\Lam$.
Recall that $\cE_{d+1}^*$ can be naturally identified with $\bigwedge_{\bR}^{d}\bR^{d+1}$ with the standard Euclidean structure. 
Under this identification, we have $\Lam^* = \bigwedge_{\bZ}^{d}\Lam$,
hence the vectors $\bigwedge_{j \neq i} \bv_j$ for $i=1,\dots,d+1$ form a basis for $\bigwedge_{\bZ}^{d}\Lam$.
We now claim that the vectors $\bigwedge_{j \neq i} \bv_j$ for $i=1,\dots,d+1$ are contained in $a_t^* N_{de^t}$ via the above identification,
which proves that $q_{d+1}(a_t \Lam) \leq ds^{-1/2}r_d$.

For each $i=1,\dots,d+1$, write
\[
\bigwedge_{j \neq i} \bv_j = \sum_{h=1}^{d+1} \left( x_h^{(i)} \bigwedge_{k \neq h} \be_k \right). 
\]
Note that $|x_{d+1}^{(d+1)}|=1/r \leq 2 \leq d$ and $x_{d+1}^{(i)}= 0$ for each $i=1,\dots,d$ since $\bv_{d+1} = r \be_{d+1}$.
By the definition of $\bv_i$ and \eqref{eqLam_d}, since $\eps<1$,
we can choose large enough $c_3 =c_3(d)>(d+1)^{d^2/7}$ such that for each $i=1,\dots,d$,
\[
\| \bv_i \| 
\leq \sqrt{1+\|v^{(i)}\|^2} 
\leq 2^{d}\sqrt{2} c_2  (\eps^{-2})^{(d-1)^2}
\leq c_3 (\eps^{-2})^{(d-1)^2}.
\]
Thus for each $i=1,\dots,d+1$ and $h=1,\dots,d$,
\[
|x_{h}^{(i)}|\leq
\left\| \bigwedge_{j \neq i} \bv_j \right\| 
\leq \prod_{j \neq i} \| \bv_j \| 
\leq c_3^d (\eps^{-2})^{d(d-1)^2}.
\]
From the assumption $c e^{-w_d t / (2 d^3)} < \eps$, it follows that
\[
c^{2 d^3} e^{-w_d t} < (\eps^{2})^{d^3} < (\eps^2)^{d(d-1)^2+1}.
\]
Choosing $c=c_3^{1/2d^2}>(d+1)^{1/14}$, we have 
\[
|x_{h}^{(i)}|  \leq c_3^d (\eps^{-2})^{d(d-1)^2} < \eps^2 e^{w_d t} = s e^{w_d t} \leq s e^{w_i t},
\]
which concludes the claim.
\end{proof}

\begin{lem}\label{CLPCSLem10}
Let $d \geq 2$, $\bbr = (r_1,\dots,r_{d+1})$, $\ov{b}_t = \diag{ \ov{b}_{t,1},\dots,\ov{b}_{t,d+1} }$, and $\Lam \in \cK^*_{\eps^2}$, where
\[
r_i = 
\begin{cases}
\eps e^{\left( \xi - \frac{1}{\ell} (w_{\ell+1} + \cdots + w_d) \right) t} & \quad \text{if}\quad 1\leq i \leq \ell, \\
\eps e^{(\xi+w_i)t} & \quad \text{if}\quad \ell+1 \leq i \leq d, \\
1 & \quad \text{if}\quad i = d+1,
\end{cases}
\]
and
\[
\ov{b}_{t,i} = 
\begin{cases}
e^{ \left( \xi w_i - \frac{1}{\ell} (w_{\ell+1} + \cdots + w_d ) \right) t} & \quad \text{if}\quad 1 \leq i \leq \ell, \\
e^{(1+\xi) w_i t} & \quad \text{if}\quad \ell+1 \leq i \leq d, \\
e^{-\xi t} & \quad \text{if}\quad i=d+1.
\end{cases}
\]
Then there exists a positive real number $\tilde{s} \leq 1$ such that for all $s,t>0$ with $e^{-\delta t} < \eps < s < \tilde{s}$ where $\delta = \frac{1}{18d^2} \min \left(\xi w_d,\xi w_1 - \frac{1}{\ell}(w_{\ell+1} + \cdots + w_d) \right)$, one has
\begin{equation}\label{eq338}
\# \cS (\ov{b}_t\Lam,\bbr,s) \leq s^2 \vol(M_\bbr).
\end{equation}
\end{lem}

\begin{proof}
Note that $r_m=r_1 <r_M = r_{\ell+1}$.
Take $t_0=t_0(w_1,\dots,w_d)>0$ such that for any $t > t_0$ we have
\begin{equation}\label{eq339}
e^{\frac{w_d}{20} t} \geq (\xi + \frac{w_d}{2}) t.
\end{equation}
Denote $c_4 = e^{-\delta t_0}$. Then $c_4 \in (0,1)$ depends only on the weights $w_1,\dots,w_d$, and 
the inequality \eqref{eq339} holds whenever $e^{-\delta t} < c_4$.
Let \eq{
\tilde{s} = \min \left( \frac{1}{C}, c_4,\frac{1}{\sqrt{d+1}},\left(\frac{\vol(B_1)}{4^{d+1}}\right)^{1/d}\right)
\leq 1.
}
By Lemma \ref{CLPCSLem8} \eqref{count_lem_2}, it suffices to show that for $e^{-\delta t} < \eps < s < \tilde{s}$,
\[
q_1(\ov{b}_t \Lam) \geq s^{-2} \quad \mathrm{and} \quad q_{d+1}(\ov{b}_t \Lam) \log q_{d+1}(\ov{b}_t \Lam) \leq s \eps e^{ (\xi + w_{\ell+1}) t }.
\]

Since $e^{-\delta t} < \eps < s$, it follows from
$s^{-3}\eps^{-1}< \eps^{-4}<e^{4\delta t}$ that $s^{-2}/r_{i} < s$ for all $i=1,\dots,d$, hence
\[
\begin{split}
\ov{b}_t^* N_{s^{-2}}&=\ov{b}_t^* M^*_{(\frac{s^{-2}}{r_1},\dots,\frac{s^{-2}}{r_d},s^{-2})}\\
&=M^*_{\left(e^{\xi (w_1-1)t} \eps^{-1} s^{-2},\dots,e^{\xi (w_d-1)t} \eps^{-1} s^{-2},e^{-\xi t} s^{-2}\right)}.
\end{split}
\]
%From the assumption $e^{-\delta t} < \eps < s$, we have
%\begin{equation}\label{eq340}
%s^{3d} \eps^{3d} > \eps^{6d} > e^{ -6d \delta t}.
%\end{equation}
Since $\tilde{s}\leq \frac{1}{\sqrt{d+1}}$, we have for all $i=1,\dots,d$,
\[
\frac{s^2 \eps^3}{\sqrt{d+1}}>\eps^{6} > e^{-6\delta t}>
e^{ -(\xi w_1 - \frac{1}{\ell}(w_{\ell+1} + \cdots + w_d) ) t}  \geq e^{\xi (w_i -1) t},
\] 
and 
\[
\frac{s^2 \eps^2}{\sqrt{d+1}} > \eps^{5} > e^{-5\delta t} > e^{-\xi w_d t}>e^{-\xi t},
\] 
hence it follows that $\ov{b}_t^* N_{s^{-2}}$ is contained in the interior of $B_{\eps^2} (\cE_{d+1}^*)$.
Since $\Lam \in \cK_{\eps^2}^*$, there is no lattice point of $\Lam$ in $\ov{b}_t^* N_{s^{-2}}$, which concludes
$q_1(\ov{b}_t \Lam) \geq s^{-2}$ as in the proof of the first inequality of \eqref{Lem9Claim}.

Since $\xi=\max(1,\frac{d-\ell}{\ell})<d$, we have
\eqlabel{eq340}{
s\eps >\eps^2 >e^{-\frac{1}{9d^2}\xi w_d t}>e^{-\frac{1}{9d} w_d t},
}
which implies that
\[
e^{\frac{w_d}{2}t}=e^{-\frac{w_d}{2}t}e^{w_d t}<e^{-\frac{1}{9d}w_d t}e^{w_{d}t}<s\eps e^{w_{d}t},
\] hence $e^{(\xi+\frac{w_d}{2})t}<r_{d}s$. On the other hand, it is clear that 
$r_\ell s< e^{(\xi+\frac{w_d}{2})t}$, hence
$\ov{b}_t^* N_{e^{(\xi+w_d/2)t}}$ is the set of $\varphi = x_1^\varphi \be_1^* + \cdots + x_{d+1}^\varphi \be_{d+1}^* \in \cE_{d+1}^*$ such that
\[
\begin{cases}
|x_i^\varphi| \leq s e^{ \left( \xi w_i - \frac{1}{\ell}(w_{\ell+1} + \cdots + w_d) \right) t } \qquad & \text{for } 1 \leq i \leq \ell, \\
|x_i^\varphi| \leq \eps^{-1} e^{ ( \xi w_i + \frac{w_d}{2} ) t } \qquad & \text{for } \ell+1 \leq i \leq d, \\
|x_i^\varphi| \leq e^{ \frac{1}{2} w_d t} \qquad& \text{for } i=d+1.
\end{cases}
\]
It follows from $\Lam \in \cK_{\eps^2}^*$ that $\lam_1(B_1,\Lam^*) \geq \eps^2$. By Minkowski's second theorem, we have
$$\eps^{2d}\lam_{d+1}(B_1,\Lam^*)\leq \lam_{1}(B_1,\Lam^*)\cdots\lam_{d+1}(B_1,\Lam^*) \leq \frac{2^{d+1}}{\vol(B_1)}.$$
Hence $\lam_{d+1}(B_1,\Lam^*)\leq \frac{2^{d+1}}{\vol(B_1)} \eps^{-2d}$. 
Thus there exists a Minkowski reduced basis $\varphi_1, \dots , \varphi_{d+1}$ of $\Lam^*$ such that 
$\| \varphi_i \| \leq \frac{4^{d+1}}{\vol(B_1)} \eps^{-2d} \leq \eps^{-3d}$ for all $i=1,\dots d+1$ 
since $\eps^d <\tilde{s}^d \leq \frac{\vol(B_1)}{4^{d+1}}$.
Recall that $w_1=\cdots=w_\ell$. It can be easily checked that 
$\varphi_i$'s are contained in $\ov{b}_t^* N_{e^{(\xi+w_d/2)t}}$.
Thus $q_{d+1} (\ov{b}_t \Lam) \leq e^{(\xi + w_d/2)t}$ so that
\begin{alignat*}{3}
q_{d+1}(\ov{b}_t \Lam) \log q_{d+1}(\ov{b}_t \Lam) 
& \leq e^{(\xi + \frac{w_d}{2})t} (\xi + \frac{w_d}{2})t  \\
& \leq e^{(\xi + \frac{w_d}{2})t} e^{\frac{w_d t}{20}} \quad &&\text{by (\ref{eq339})} \\
& \leq s\eps e^{(\xi + w_d)t} \quad &&\text{by (\ref{eq340})}\\
& \leq s\eps e^{(\xi + w_{\ell+1})t}.
\end{alignat*}
\end{proof}

% Section : Lower bound %%%%%%%%%%%%%%%%%%%%%%%%%%%%%%%%%%%%%%%%%%%%%%%%%%%%%%%%%%%%%%%%%%%%%%%%%%%%

\section{Lower bound}\label{sec4}

% subsection : Construction of the fractal set %%%%%%%%%%

\subsection{Construction of the fractal set}\label{sec4.1}
For a given weight vector $w=(w_1,\dots,w_d)$, recall that $1\leq \ell \leq d-1$ is the unique integer such that 
$w_1 = \cdots = w_{\ell} > w_{\ell+1} \geq \cdots \geq w_d$, 
and $\xi = \max(1,\frac{d-\ell}{\ell})$ (see \textsection \ref{subsec3.2}).
Since $1<\zeta(d+1)<2$, we can choose a real number $c_0>1$ such that
\eqlabel{c0eq}{
\frac{1}{10}<\left(\frac{2}{c_0}-c_0\right) \frac{1}{\zeta(d+1)}\quad\text{and}\quad\frac{c_0}{\zeta(d+1)}<1.
} 
Let $\tilde{c}\leq 1$ be a positive real number as in Lemma \ref{CLPCSLem7} with respect to the above $c_0$, and let $\tilde{\eps}, \tilde{s} \leq 1$ be positive real numbers as in Lemmas \ref{CLPCSLem9} and \ref{CLPCSLem10}, respectively.
We fix the constants $\eps, t, r >0$ with the following properties:
\begin{enumerate}
\item\label{epsrprop} $0<\eps<r<\frac{1}{10^4 4^d}\min\{\tilde{\eps}, \tilde{s} \}$;
\item $t\geq 1$ will be chosen large enough so that \eqref{t1}, \eqref{t2}, \eqref{t3}, \eqref{t4}, \eqref{t5}, \eqref{t6}, \eqref{t7}, \eqref{t8} hold.
\end{enumerate}
Let $\{\eps_n\}$ and $\{t_n\}$ be the sequence defined as follows: for $n \in \bN$,
\begin{enumerate}
\item $\eps_n = \eps/n$;
\item $t_n - t_{n-1} = \xi nt$ and $t_0=0$.
\end{enumerate}

We will construct the tree $\cT$ whose vertices are in the set $\bQ^d$ of rational vectors and the map $\beta$ from the set of vertices of $\cT$ to the set of compact subsets in $\bR^{d}$, inductively.
We first set the root of $\cT$ to be zero, that is, $\tau_0=\mb{0}$ and define
\eq{ \beta(\tau_0) = \{ x\in\bR^{d} : |(\tau_0)_i - x_i| < e^{-w_i t_1}, \forall i=1,\dots,d \}. }
For each $\tau \in \cT_{n}$ with $n\geq 1$, let
\[
\tilde{\beta}(\tau) = \{ x \in \bR^{d} : |\tau_i - x_i| < \eps_{n+1} e^{-w_i t_{n+1} - t_{n}}, \forall i=1,\dots,d \}.
\]
Recall that $a_t = \diag{e^{w_1 t},\dots,e^{w_d t},e^{-t}}$ and
$h(x)=\mat{I_d & x \\ 0 & 1}$ for $x\in\bR^d$.
Denote by  
\[
b_n = \text{diag} ( e^{ -\frac{1}{\ell}(w_{\ell+1} + \cdots + w_d) nt }, \dots, e^{ -\frac{1}{\ell}(w_{\ell+1} + \cdots + w_d) nt }, e^{w_{\ell+1} nt}, \dots, e^{w_d nt}, 1 ).\]
Note that the first $\ell$ terms of $b_n$ are the same.

For each $\kappa\in \cT_{n-1}$, we define $\cT(\kappa)$ as the set of all $\tau \in \tilde{\beta}(\kappa)$ with the following properties:
\eqlabel{treecondition}{
\begin{split}
a&_{t_n} h(\tau) \bZ^{d+1} \in \cL_{d+1}', \\
a&_{t_n} h(\tau) \bZ^{d+1} \in \cK_{\eps_n^2}^*, \\
b&_n a_{t_n} h(\tau) \bZ^{d+1} \in \cK_r^*.
\end{split}
}
We use the first property to show that the self-affine structure  is indeed a subset of the set of weighted singular vectors.
The second property is intended to estimate the number of sons of given vertices.
Finally, the third property is used to show that the rectangles in the self-affine structure are well-separated.

It follows from the definitions of $\cT(\kappa)$ and $\cL_{d+1}'$ that $\tau \in \bQ^d$ and for any $\tau \in \cT(\kappa)$ there exists a unique vector
\eqlabel{eqvec}{ \bv(\tau) \in \{r\be_{d+1} : 1/2 < r \leq 1 \} \cap a_{t_n}h(\tau)\bZ^{d+1}. }
Note that $(d+1)$-th coordinate of $\bv(\tau)$ is $q e^{-t_n}$ for some $q \in \bZ$ such that $1/2 < qe^{-t_n} \leq 1$.
Since $t_n \geq t_{n-1} + 1$, $\cT_n$ has empty intersection with $\bigcup_{0 \leq i \leq n-1} \cT_i$, which
implies that $\cT$ is a rooted tree.

For each $\tau \in \cT(\kappa)$ with $\kappa \in \cT_{n-1}$, define
\eq{ \beta(\tau) = \{ x\in\bR^{d} : |\tau_i - x_i| < \eps_{n} e^{-w_i t_{n+1}- t_{n}}, \forall i=1,\dots,d \}. }
Note that for each $\tau \in \cT(\kappa)$, it follows from the definitions of $\tilde{\beta}$ and $\beta$ that
$\beta(\tau)\subset \beta(\kappa)$.  
If follows from Lemma \ref{LBLem1} below that each vertex of $\cT$ has sons by choosing $t\geq 1$ large enough so that for any $n\in\bN$
\eqlabel{t1}{\frac{1}{100} \eps_n^d e^{\xi d nt} \geq 1.} 
Hence the pair $(\cT,\beta)$ is a regular self-affine structure.

\begin{lem}\label{LBLem1}
For every $n \in \bN$ and $y \in \cT_{n-1}$ one has
\eq{ \frac{1}{100} \eps_n^d e^{\xi d nt} \leq \#\cT(y) \leq 2^{d+1} \eps_n^d e^{\xi d nt}. }
\end{lem}

For fixed $n \in \bN$ and $y \in \cT_{n-1}$, we let
\begin{align*}
&\Lambda = a_{t_{n-1}} h(y)  \bZ^{d+1} \in \cL'_{d+1} \cap \cK_{\eps_{n-1}^2}^*, \\
&\Lambda_1 = a_{t_n} h(y)  \bZ^{d+1} = a_{\xi nt}\Lam, \\
&\Lambda_2 = b_n a_{t_n} h(y)  \bZ^{d+1} = b_n a_{\xi nt} \Lam,
\end{align*}
 and for $x \in \tilde{\beta}(y)$,
\begin{align*}
&\Lambda_1(x) := a_{t_n} h(x)  \bZ^{d+1} = a_{t_n} h(x-y) a_{t_n}^{-1} \Lam_1, \\
&\Lambda_2(x) := b_n a_{t_n} h(x)  \bZ^{d+1} = b_n a_{t_n} h(x-y) a_{t_n}^{-1} b_n^{-1} \Lam_2.
\end{align*}

The lattices $\Lambda_1(x)$ and $\Lambda_2(x)$ satisfy $\Lambda_1(x) \in \cL_{d+1}' \cap \cK_{\eps_n^2}^*$ and $\Lambda_2(x) \in \cK_r^*$ if and only if $x \in \cT(y)$.
Hence Lemma \ref{LBLem1} follows from the following lemma.

\begin{lem}\label{LBLem2}
Let $n \in \bN$ and $y \in \cT_{n-1}$. Then
\begin{align}
\frac{1}{10} \eps_n^d e^{ \xi d nt} \leq \#\{x \in \tilde{\beta}(y) : \Lambda_1(x) \in \cL_{d+1}' \} \leq 2^{d+1} \eps_n^d e^{\xi d nt}, \label{eq412} \\
\#\{x \in \tilde{\beta}(y) : \Lambda_1(x) \in \cL_{d+1}' \smallsetminus \cK_{\eps_n^2}^*  \} \leq \frac{8}{100}\eps_n^d e^{\xi d nt}, \label {eq413}\\
\#\{x \in \tilde{\beta}(y) : \Lambda_2(x) \in \cL_{d+1}' \smallsetminus \cK_r^*  \} \leq \frac{1}{100}\eps_n^d e^{\xi d nt}. \label{eq414}
\end{align}
\end{lem}

\begin{proof} % later %%%%%%%%%%%%%%%%%%%%%%%%%%%
Let $x \in \tilde{\beta}(y)$ with $\Lam_1(x) \in \cL_{d+1}'$.
Then there exists $s_x$ such that $1/2 < s_x \leq 1$ and $\Lam_1(x) \cap \bR \be_{d+1} = \{s_x \be_{d+1} \bZ\}$.
We denote $s_x \be_{d+1}$ by $\bv (x)$.

First, we prove (\ref{eq412}).
It can be checked by a direct calculation that the map $x \mapsto a_{t_n} h(y-x) a_{t_n}^{-1} \bv (x)$ is a bijection from $\{x \in \tilde{\beta}(y) : \Lambda_1(x) \in \cL_{d+1}' \}$ to $M \cap \wh{\Lam}_1$ where
\[
M = \{ (z_1,\dots, z_{d+1}) : \max_{1\leq i\leq d} |z_i| \leq \eps_n e^{\xi nt} |z_{d+1}|,\ 1/2 < |z_{d+1}| \leq 1 \}.
\]
Thus it suffices to estimate $\# (M \cap \wh{\Lam}_1)$.
Let
\begin{align*}
M^{(1)} &= \{ (z_1,\dots,z_{d+1}) : \max_{1\leq i\leq d} |z_i| \leq \frac{1}{2} \eps_n e^{\xi nt},\ |z_{d+1}| \leq 1 \} \\
M^{(2)} &= \{ (z_1,\dots,z_{d+1}) : \max_{1\leq i\leq d} |z_i| \leq \frac{1}{2} \eps_n e^{\xi nt},\ |z_{d+1}| \leq \frac{1}{2} \}.
\end{align*}
Since $M^{(1)} \smallsetminus M^{(2)} \subset M \subset 2 M^{(2)}$, we have
\begin{equation}\label{eq415}
\# (M^{(1)} \cap \wh{\Lam}_1) - \# (M^{(2)} \cap \wh{\Lam}_1) \leq \# (M \cap \wh{\Lam}_1) \leq \# (2M^{(2)} \cap \wh{\Lam}_1).
\end{equation}

We will use Lemma \ref{CLPCSLem7} to estimate $\# (M^{(i)} \cap \wh{\Lam}_1)$ for $i=1,2$.
Since $\Lam \in \cK_{\eps_{n-1}^2}^*\subset \cK_{\eps_{n}^2}^*$, it follows from the natural identification $\cE_d^{**}=\bR^d$ and Minkowski's second theorem that there exist contants $C_1, C_2>0$ depending only on $d$ such that
\[
\lam_1(B_1,\Lam) \geq C_1 \eps_n^{2d} \text{ and } \lam_{d+1}(B_1,\Lam) \leq C_2 \eps_n^{-2}.
\]
Since $\Lam = a_{\xi nt}^{-1}\Lam_1$, for $i=1,2$, we have
\[
\begin{split}
\lam_1(M^{(i)},\Lam_1) 
&= \lam_1( a_{\xi nt}^{-1} M^{(i)},\Lam) 
\geq \lam_1( a_{\xi nt}^{-1}M^{(1)},\Lam)\\
&\geq \lam_1( B_{(d+1) e^{\xi nt}},\Lam) 
\geq \frac{C_1}{d+1} e^{-\xi nt} \eps_n^{2d}
\end{split}
\]
and
\[
\begin{split}
\lam_{d+1} ( M^{(i)}, \Lam_1 )
&= \lam_{d+1} ( a_{\xi nt}^{-1} M^{(i)}, \Lam )
\leq \lam_{d+1} ( a_{\xi nt}^{-1} M^{(2)}, \Lam )\\
&\leq \lam_{d+1} ( B_{\frac{1}{2} \eps_n e^{(1-w_1) \xi nt}}, \Lam )
\leq 2C_2 e^{(w_1-1) \xi nt} \eps_n^{-3}.
\end{split}
\]
Thus we can choose $t\geq 1$ large enough so that for all $n\in\bN$
\eqlabel{t2}{
\lam_{d+1} ( M^{(i)}, \Lam_1 ) <\wt{c} \quad \text{and}\quad
-\lam_{d+1} ( M^{(i)}, \Lam_1 ) \log \lam_{1} ( M^{(i)}, \Lam_1 ) <\wt{c}. } 
Using Lemma \ref{CLPCSLem7} and \eqref{eq415}, we have
\[
\left( \frac{2}{c_0} - c_0 \right) \frac{1}{\zeta(d+1)} \eps_n^d e^{\xi dnt}
\leq \# ( M \cap \wh{\Lam}_1 )
\leq \frac{c_0}{\zeta(d+1)} 2^{d+1} \eps_n^d e^{\xi dnt}.
\] 
By \eqref{c0eq}, we complete the proof of \eqref{eq412}.

Next, we prove (\ref{eq413}) and (\ref{eq414}).
Let $s_1 = \eps_n^2$, $s_2 = r$, $a^{(1)} = a_{\xi nt}$, $a^{(2)} = b_n a_{\xi nt}$, and 
\[
\cS_j = \{ x \in \tilde{\beta}(y) : \Lam_j(x) \in \cL_{d+1}' \smallsetminus \cK_{s_j}^* \} \text{ for } j=1,2.
\]
Recall that 
\eq{
\cS(\Lam,\bbr,s) = \left\{ \bv \in M_\bbr \cap \widehat{\Lam} : \varphi(\bv) = 0 \text{ for some } \varphi \in N_{(d+1)sr_M}(\bbr,s) \cap \widehat{\Lam^*} \right\}.
}
We will show that 
\begin{equation} \label{eqmine}
\# \cS_j \leq \# \cS(\Lam_j, \bbr_j, s_j) \quad (j=1,2)
\end{equation}
for some $\bbr_j$ and apply Lemma \ref{CLPCSLem9} and \ref{CLPCSLem10}.

Let $z^{(1)}$ and $z^{(2)}$ be vectors in $\bR^d$ such that
\begin{align*}
z^{(1)}_i &= (y_i - x_i) e^{(w_i + 1)t_n} \quad \text{for } 1 \leq i \leq d;\\
z^{(2)}_i &= 
\begin{cases}
(y_i - x_i) e^{- \frac{1}{\ell}(w_{\ell+1} + \cdots + w_d) nt + (w_i +1)t_n} \quad &\text{if } 1 \leq i \leq \ell, \\
(y_i - x_i) e^{w_i nt + (w_i + 1)t_n} \quad &\text{if } \ell+1 \leq i \leq d.
\end{cases}
\end{align*}
Then $h(z^{(j)}) = a^{(j)} a_{t_{n-1}} h(y-x) (a^{(j)} a_{t_{n-1}})^{-1}$ for $j=1,2$ and
\begin{align*}
|z^{(1)}_i| &\leq \eps_n e^{\xi nt} =: r^{(1)}_{i} \quad \text{for } 1 \leq i \leq d; \\
|z^{(2)}_i| &\leq 
\begin{cases}
\eps_n e^{\left( \xi - \frac{1}{\ell}(w_{\ell+1} + \cdots + w_d) \right) nt} =: r^{(2)}_{i} \quad &\text{if } 1 \leq i \leq \ell, \\
\eps_n e^{(\xi + w_i)nt} =: r^{(2)}_{i} \quad &\text{if } \ell+1 \leq i \leq d.
\end{cases}
\end{align*}
Since $\bv(x) \in \Lam_1(x) \cap \Lam_2(x)$, for $j=1,2$,
\[
\bw_j(x) := h(z^{(j)}) \bv(x) \in \Lam_j.
\]
For $\bbr_j = (r^{(j)}_{1},\dots,r^{(j)}_{d},1)$, let $x \in \cS_j$.
Then, for $i=1, \dots, d$, the $i$-th component of $\bw_j(x)$ is $s_x z_i^{(j)}$ whose norm is smaller than $r_i^{(j)}$ and the $(d+1)$-th component of $\bw_j(x)$ is 1.
Thus we have $\bw_j(x) \in M_{\bbr_j}$.
If $\bw_j(x)$ is not primitive, then there exists an interger $n'>1$ such that $\frac{1}{n'}\bw_j(x) \in \Lam_j$.
Then we have $\frac{1}{n'} \bv(x) = \frac{1}{n'} h(-z^{(j)})\bw_j(x) \in \Lam_1(x) \cap \Lam_2(x)$, which contradicts to the definition of $\bv(x)$ and $\cL_{d+1}'$.
Thus we have $\bw_j(x) \in M_{\bbr_j} \cap \wh{\Lam}_j$.
Since for $x \in \cS_j$, the vector $\bv(x)$ is uniquely determined and the multiplication by $h(z^{(j)})$ is an invertible linear transformation, the map 
\[\cS_j \to M_{\bbr_j} \cap \wh{\Lam}_j \quad \text{given by} \quad x \mapsto \bw_j(x)\] is well-defined and injective. 

Hence, in order to show (\ref{eqmine}), we should find $\varphi_j \in N_{(d+1) s_j r^{(j)}_M}(\mb{r}_j,s_j) \cap \wh{\Lam_j^*}$ such that $\varphi_j (\bw_j(x))=0$.
It follows from the definition of $\cS_j$ that for $x \in \cS_j$, $a^{(j)} a_{t_{n-1}} h(x) \bZ^{d+1} \notin \cK_{s_j}^*$.
Then there exists $\varphi_j \in \wh{\Lam^*_j}$ such that $\|h( z^{(j)} )^* \varphi_j \| < s_j$, where $h( z^{(j)} )^*$ is the adjoint action defined by $g^*\varphi(\bv) = \varphi(g\bv)$ for all $g \in \SL_{d+1}(\bR)$, $\varphi \in \cE_{d+1}^*$, and $\bv \in \bR^{d+1}$.
It follows from direct calculation that 
\[
h(z^{(j)})^* \varphi_j = \left( \varphi_j (\be_1), \dots, \varphi_j (\be_d), \sum_{i=1}^d z_i^{(j)} \varphi_j (\be_i) + \varphi_j(\be_{d+1}) \right).
\]
Choose $t\geq 1$ large enough so that for all $n\in\bN$ 
\eqlabel{t3}{\eps e^{\xi nt} \geq 1.} It follows from
$\|h( z^{(j)} )^* \varphi_j \| < s_j$ that
\begin{align*}
&|\varphi_j (\be_i)| < s_j \quad \text{for } 1 \leq i \leq d ;\\
&|\varphi_j (\be_{d+1})| < s_j + ds_j r^{(j)}_{M} < (d+1) s_j r^{(j)}_{M}.
\end{align*}
Hence we have $\varphi_j \in N_{(d+1) s_j r^{(j)}_{M}}$.
It follows that 
\begin{align*}
| \varphi_j (\bw_j(x)) |
&= | h(z^{(j)})^* \varphi_j ( h( -z^{(j)} ) \bw_j(x) ) |
= | h(z^{(j)})^* \varphi_j ( \bv(x) ) | \\
&\leq | h(z^{(j)})^* \varphi_j ( \be_{d+1} ) |
\leq \| h(z^{(j)})^* \varphi_j \|
<s_j
<1.
\end{align*}
Since $\varphi_j ( \bw_j(x) ) \in \bZ$, it follows that $\varphi_j ( \bw_j(x) ) = 0$. This proves \eqref{eqmine}.

We choose $t\geq 1$ large enough so that for all $n\in\bN$ 
\eqlabel{t4}{ce^{-w_d \xi n t / (2 d^3)} < \eps_n \quad \text{and}\quad e^{-\delta n t} < \eps_n .}
Since $\Lam \in \cK_{\eps_{n-1}^2}^*\subset \cK_{\eps_{n}^2}^*$ and by \eqref{t4}, it follows from Lemmas \ref{CLPCSLem9} and \ref{CLPCSLem10} that
\begin{align*}
&\cS_1 \leq \# \cS(\Lam_1,\bbr_1,s_1) \leq \eps_n \vol(M_{\bbr_1}) = 2^{d+1} \eps_n^{d+1} e^{\xi dnt}, \\
&\cS_2 \leq \# \cS(\Lam_2,\bbr_2,s_2) \leq r^2 \vol(M_{\bbr_2}) = 2^{d+1} r^2 \eps_n^d e^{\xi dnt}.
\end{align*}
By the assumption \eqref{epsrprop} for $\eps$ and $r$, this completes the proof.
\end{proof}

The following lemma is a $d$-dimensional version of \cite[Lemma 4.1]{LSST}. 
\begin{lem}\label{LBThm1}
$ \cF(\cT,\beta) \subset \mathrm{Sing}(w)$.
\end{lem}

\begin{proof} % later %%%%%%%%%%%%%%%%%%%%%%%%%%% 근데 이건 안해도 될거같은데...
This lemma directly follows from the same argument in the proof of \cite[Lemma 4.1]{LSST}.
\end{proof}

% subsection : The lower bound calculation 	%%%%%%%%%%

\subsection{The lower bound calculation}
In this subsection we complete the proof of main results.
\begin{prop}\label{LBProp1}
Let $w=(w_1,\dots,w_d)\in \bR_{>0}^d$ where $w_1 = \cdots = w_\ell > w_{\ell+1} \geq \cdots \geq w_d >0$ and $\sum_{i=1}^d w_i = 1$ and let $(\cT,\beta)$ be the self-affine strunction on $\bR^d$ in the previous section. Then
\eq{ \dim \cF(\cT,\beta) \geq d-\frac{1}{1+w_1}. }
\end{prop}

We will prove Proposition \ref{LBProp1} using Corollary \ref{FHCor24}. %cor 2.4
Let $C_n, L_n^{(1)}, \dots, L_n^{(d)}$ be the positive constants defined as follows:
\eq{ C_n = \eps_n^d e^{\xi d nt}, \quad L_n^{(i)} = 2 \eps_n e^{-w_i t_{n+1} - t_n}, \forall i=1,\dots,d. }
It can be easily checked that the regular self-affine structure $(\cT,\beta)$ satisfies the assumptions \eqref{thm2pro1}, \eqref{thm2pro2}, and \eqref{thm2pro3} of 
Theorem \ref{FHThm1}. For the assumption \eqref{thm2pro4} of Theorem \ref{FHThm1}, we need the following lemma.

\begin{lem}\label{LBLem2}
Let $n\in\bN$ be large and $\tau \in \cT_{n-1}$. Then
\eq{ \mathrm{dist}(\beta(x),\beta(y)) \geq L_{n-1}^{(1)}\frac{c'r^{d-1}}{4\sqrt{d}\eps_{n-1}} 
e^{\frac{1}{\ell}(w_{\ell+1}+\cdots+w_d-\xi\ell)nt},
}
where $x,y \in \cT(\tau)$ are distinct and $c'$ is the positive constant in Lemma \ref{LBLem3}.
\end{lem}

\begin{proof}
By the construction of $\cT$ and the definition of $b_n$, there are $1/2 \leq s_x, s_y  \leq 1$ such that
\eq{ s_x\be_{d+1} \in b_n a_{t_n} h(x) \bZ^{d+1} , \quad s_y\be_{d+1} \in b_n a_{t_n} h(y) \bZ^{d+1}. }
Let us denote 
\eq{
\begin{split} 
\bv &= b_n a_{t_n} h(y-x)(b_n a_{t_n})^{-1} s_x \be_{d+1} \in b_n a_{t_n} h(y) \bZ^{d+1},\\
\bv &\wedge s_y \be_{d+1} = s_x s_y \sum_{i=1}^{d} u_i\be_i \wedge \be_{d+1}.
\end{split}
} 
Observe that 
\[
u_i = 
\begin{cases} 
(y_i - x_i) e^{(w_i +1) t_n - \frac{1}{\ell}(w_{\ell+1} + \cdots + w_d)nt} \quad &\text{for } 1\leq i\leq \ell,\\
(y_i - x_i) e^{(w_i +1) t_n + w_i nt}\quad  &\text{for } \ell+1 \leq i\leq d.
\end{cases}
\]
Since $x$ and $y$ are distinct, the vectors $\bv$ and $e_y \be_{d+1}$ are linearly independent,
hence it follows from Lemma \ref{LBLem3} that 
\eqlabel{wedgebound}{
\sqrt{d}\|\mb{u}\|_{\infty} \geq s_x s_y \| \bu \| = \| \bv \wedge s_y \be_{d+1} \| \geq c' r^{d-1},
}
where $\bu=(u_1,\dots,u_d)\in\bR^d$ and $\|\cdot\|_\infty$ denotes the max norm.

Let $x' \in \beta(x)$ and $y' \in \beta(y)$.
Suppose that $\|\mb{u}\|_\infty = |u_i|$ for some $1\leq i\leq \ell$. 
Then it follows from \eqref{wedgebound} that 
\begin{align}
\|y'-x'\| &\geq |y_i'-x_i'| \geq |y_i-x_i| - |x_i-x_i'| - |y_i-y_i'| \nonumber\\ 
&\geq e^{-(w_i+1) t_n + \frac{1}{\ell}(w_{\ell+1}+\cdots+w_d) nt } \left( \frac{c' r^{d-1}}{\sqrt{d}} - 2\eps_n e^{-\xi w_i (n+1) t - \frac{1}{\ell}(w_{\ell+1} + \cdots + w_d)nt} \right) \nonumber\\
&\geq e^{-(w_i+1) t_n + \frac{1}{\ell}(w_{\ell+1}+\cdots+w_d) nt } \frac{c' r^{d-1}}{2\sqrt{d}} \label{t5} \\
&= L_{n-1}^{(i)} \frac{c' r^{d-1}}{4\sqrt{d}\eps_{n-1}} e^{\frac{1}{\ell}(w_{\ell+1}+\cdots+w_d-\xi \ell) nt} \nonumber\\
&\geq L_{n-1}^{(1)} \frac{c' r^{d-1}}{4\sqrt{d}\eps_{n-1}} e^{\frac{1}{\ell}(w_{\ell+1}+\cdots+w_d-\xi \ell) nt}. \nonumber
\end{align}
We choose $t\geq 1$ large enough so that the third line \eqref{t5} holds for all large enough $n\in\bN$.

On the other hand, if $\|\mb{u}\|_\infty = |u_i|$ for some $\ell +1 \leq i\leq d$, then
we have
\begin{align}
\|y'-x'\| &\geq |y_i'-x_i'| \geq |y_i-x_i| - |x_i-x_i'| - |y_i-y_i'| \nonumber\\ 
&\geq e^{-(w_i+1) t_n -w_i nt } \left( \frac{c' r^{d-1}}{\sqrt{d}} - 2\eps_n e^{w_i nt-\xi w_i (n+1) t} \right) \nonumber\\
&\geq e^{-(w_i+1) t_n -w_i nt } \frac{c' r^{d-1}}{2\sqrt{d}} \label{t6}\\
&= L_{n-1}^{(1)} \frac{c' r^{d-1}}{4\sqrt{d}\eps_{n-1}} e^{(w_1-w_i)t_n - (\xi+w_i)nt} \nonumber\\
&\geq L_{n-1}^{(1)} \frac{c' r^{d-1}}{4\sqrt{d}\eps_{n-1}} e^{(w_1-w_{\ell+1})t_n - (\xi+w_i)nt} \nonumber\\
&\geq L_{n-1}^{(1)} \frac{c' r^{d-1}}{4\sqrt{d}\eps_{n-1}} e^{\frac{1}{\ell}(w_{\ell+1}+\cdots+w_d-\xi \ell) nt}. \label{t7}
\end{align} 
We choose $t\geq 1$ large enough so that the third line \eqref{t6} and last line \eqref{t7} hold for all large enough $n\in\bN$.

This concludes the proof of the lemma.
\end{proof}

Since $w_{\ell+1}+\cdots+w_d < \xi\ell$, we can choose $t\geq 1$ large enough so that for all $n\in\bN$
\eqlabel{t8}{\rho_n := \frac{c'r^{d-1}}{4\sqrt{d}\eps_{n-1}}e^{\frac{1}{\ell}(w_{\ell+1}+\cdots+w_d-\xi\ell)nt} \leq 1.} 
The assumption \eqref{thm2pro4} of Theorem \ref{FHThm1} follows from Lemma \ref{LBLem2}. 
\begin{proof}[Proof of Proposition \ref{LBProp1}]
We prove the proposition applying Corollary \ref{FHCor24}. %cor 2.4
We first claim that the assumptions of Corollary \ref{FHCor24} hold for some large $k, n_0 \in \bN$.
%Let $n_0 \in\bN$ be such that Lemma \ref{LBLem2} holds for all $n \geq n_0$.  
To show \ref{FHCorA1}, observe that $\frac{ L_{kn}^{(d)} }{ L_{kn-1}^{(d)} } \leq \frac{ L_n^{(1)} }{L_{n-1}^{(1)}}$ if and only if
\[
\frac{kn-1}{kn}\cdot \frac{n}{n-1} \leq e^{w_d \xi (kn+1)t - w_1 \xi(n+1)t + \xi knt -\xi nt}
\]
If $k\in\bN$ is large enough, then the right hand side of the above equation has a positive exponent, hence $\frac{ L_{kn}^{(d)} }{ L_{kn-1}^{(d)} } \leq \frac{ L_n^{(1)} }{L_{n-1}^{(1)}}$, for all large enough $n\in\bN$. Similarly, $L_{kn_0-1}^{(d)} < L_{n_0-1}^{(1)}$ also holds for some large $k, n_0 \in\bN$. Since $C_n$ has a positive exponent and $\rho_n $ has a negative exponent, we can choose $k\in\bN$ large enough so that \ref{FHCorA2} and \ref{FHCorA3} hold for all large enough $n\in\bN$. To show \ref{FHCorA4}, observe that for all large enough $n\in\bN$,
\[
\rho_n^\ell C_n \prod_{j=\ell+1}^d L_n^{(j)}/L_{n-1}^{(j)}=const \times \frac{\eps_n^{2d-\ell}}{\eps_{n-1}^{d}} \geq const \times \eps^{d-\ell}2^{-d}n^{-d+\ell}, 
\] where the constant is independent of $n$. Thus we can choose $k\in\bN$ large enough so that \ref{FHCorA4} holds for all large enough $n\in\bN$. This proves the claim.

%\item\label{FHCorA1} $\frac{ L_{kn}^{(d)} }{ L_{kn-1}^{(d)} } \leq \frac{ L_n^{(1)} }{L_{n-1}^{(1)}} \text{ and } L_{kn_0-1}^{(d)} < L_{n_0-1}^{(1)},$
%\item\label{FHCorA2} $e^{n/k} \leq C_n \leq e^{kn},$
%\item\label{FHCorA3} $e^{-kn} \leq \rho_n \leq e^{-n/k},$
%\item\label{FHCorA4} $\rho_n^\ell C_n \prod_{j=\ell+1}^d L_n^{(j)}/L_{n-1}^{(j)} \geq n^{-k}.$
 
Then we have
\begin{align*}
&\frac{\log(C_n L_n^{(\ell+1)} \cdots L_n^{(d)}/L_{n-1}^{(\ell+1)} \cdots L_{n-1}^{(d)})}{-\log(L_n^{(1)}/L_{n-1}^{(1)})} \\
&= \frac{\xi d nt - \xi (w_{\ell+1}+\cdots+w_d)(n+1)t - \xi (d-\ell) nt + o(n)}{ \xi w_1(n+1)t + \xi nt + o(n)} \\
&\to \frac{\ell-(w_{\ell+1}+\cdots+w_d)}{1+w_1} = \ell-\frac{1}{1+w_1} \quad\text{ as } n \to \infty
\end{align*}
Hence Corollary \ref{FHCor24} implies 
\eq{ \dim \cF(\cT,\beta) \geq (d-\ell)+\ell-\frac{1}{1+w_1} = d -\frac{1}{1+w_1}. }
\end{proof}

\begin{proof}[Proof of Theorem \ref{mainthm}]
If $w_1=\cdots=w_d$, then the result follows from \cite[Theorem 1.1]{CC16}. % CC16
If there exists $1 \leq \ell \leq d-1$ such that $w_1 = \cdots = w_\ell > w_{\ell+1} \geq \cdots \geq w_d$, then the result follows from Lemma \ref{LBThm1} and Proposition \ref{LBProp1}.
\end{proof}

\begin{proof}[Proof of Theorem \ref{mainthm_2}]
It follows from the same argument in the proof of \cite[Theorem 1.5]{LSST}.
\end{proof}
\def\cprime{$'$} \def\cprime{$'$} \def\cprime{$'$}
\providecommand{\bysame}{\leavevmode\hbox to3em{\hrulefill}\thinspace}
\providecommand{\MR}{\relax\ifhmode\unskip\space\fi MR }
% \MRhref is called by the amsart/book/proc definition of \MR.
\providecommand{\MRhref}[2]{%
  \href{http://www.ams.org/mathscinet-getitem?mr=#1}{#2}
}
\providecommand{\href}[2]{#2}

\end{document}